\documentclass[11pt]{amsart}

\usepackage{amsmath}
\usepackage{amsfonts}
\usepackage{amssymb}
\usepackage{amsthm}
\usepackage{graphicx}
\usepackage{hyperref}
\usepackage{pb-diagram}
\usepackage{epstopdf}
\usepackage{amscd}
\usepackage{pdfpages}
\usepackage{bbm}
\usepackage{multirow}
\usepackage[mathscr]{eucal}
\usepackage{epsfig,epsf,epic}
\usepackage{pdfsync}
\usepackage{enumerate}
\usepackage{etex}
\usepackage[all]{xy}
\usepackage{fancyref}
\usepackage{mathtools}
\usepackage{scalerel,stackengine}
\usepackage{comment}
\stackMath

\newtheorem{theorem}{Theorem}[section]

\newtheorem{corollary}[theorem]{Corollary}
\newtheorem{definition}[theorem]{Definition}
\newtheorem{example}[theorem]{Example}
\newtheorem{lemma}[theorem]{Lemma}
\newtheorem{remark}[theorem]{Remark}

\newtheorem{notation}[theorem]{Notations}
\numberwithin{equation}{section}

\DeclareMathOperator{\Hom}{Hom}

\DeclareMathOperator{\Spec}{Spec}
\DeclareMathOperator{\supp}{supp}

\DeclareMathOperator{\vol}{vol}

\newcommand{\real}{\mathbb{R}}
\newcommand{\comp}{\mathbb{C}}

\newcommand{\inte}{\mathbb{Z}}

\newcommand{\dd}[1]{\frac{\partial}{\partial #1}}

\newcommand{\half}{\frac{1}{2}}

\newcommand{\reallywidehat}[1]{%
	\savestack{\tmpbox}{\stretchto{%
			\scaleto{%
				\scalerel*[\widthof{\ensuremath{#1}}]{\kern-.6pt\bigwedge\kern-.6pt}%
				{\rule[-\textheight/2]{1ex}{\textheight}}
			}{\textheight}%
		}{0.5ex}}%
	\stackon[1pt]{#1}{\tmpbox}%
}

\newcommand{\reallywidetilde}[1]{%
	\savestack{\tmpbox}{\stretchto{%
			\scaleto{%
				\scalerel*[\widthof{\ensuremath{#1}}]{\kern-.6pt\sim\kern-.6pt}%
				{\rule[-\textheight/2]{1ex}{\textheight}}
			}{\textheight}%
		}{0.5ex}}%
	\stackon[1pt]{#1}{\tmpbox}%
}

\newcommand{\dr}[1]{\text{dR}_{#1}(M)}
\newcommand{\eqopera}{\mathcal{P}}
\newcommand{\actionmap}{\sigma}
\newcommand{\simplex}{\blacktriangle}
\newcommand{\lam}{\lambda}
\newcommand{\tr}{T}
\newcommand{\tree}[1]{\mathbb{T}_{#1}}
\newcommand{\ltree}[1]{\mathbb{L}\mathbb{T}_{#1}}
\newcommand{\mcpx}[1]{\text{CM}_{#1}}
\newcommand{\mtr}{\mathcal{T}}
\newcommand{\val}{\nu}
\newcommand{\lvertex}{L\tr^{[0]}}

\newcommand{\ftree}{\Gamma}
\newcommand{\morseprod}[1]{m_{#1}^{\text{eMorse}}}

\newcommand{\dist}{d}
\newcommand{\tdist}{\rho}
\newcommand{\g}{g}

\begin{document}

\title[Fukaya's conjecture on $S^1$-equivariant de Rham complex]{Fukaya's conjecture on $S^1$-equivariant de Rham complex}

\author[Ma]{Ziming Nikolas Ma}
\address{The Institute of Mathematical Sciences and Department of Mathematics\\ The Chinese University of Hong Kong\\ Shatin \\ Hong Kong}
\email{zmma@math.cuhk.edu.hk}
\email{nikolasming@outlook.com}

\begin{abstract}
	Getzler-Jones-Petrack \cite{getzler1991differential} introduced $A_\infty$ structures on the equivariant complex for manifold $M$ with smooth $\mathbb{S}^1$ action, motivated by geometry of loop spaces. Applying Witten's deformation by Morse functions followed by homological perturbation we obtained a new set of $A_\infty$ structures. We extend and prove Fukaya's conjecture \cite{fukaya05} relating this Witten's deformed equivariant de Rham complexes, to a new Morse theoretical $A_\infty$ complexes defined by counting gradient trees with jumping which are closely related to the $\mathbb{S}^1$ equivariant symplectic cohomology proposed by Siedel \cite{seidel2007biased}. 
	\end{abstract}

\maketitle

\section{Introduction}\label{sec:introduction}

In the influential paper \cite{witten82} by Witten, harmonic forms on a compact oriented Riemannian manifold $(M,g)$ are related to the Morse complex $CM^*_f:=\bigoplus_{p\in \text{Crit}(f)} \comp \cdot p $ on $M$ with a Morse function $f$ \footnote{Here $\text{Crit}(f)$ refers to set of critical points of $f$, and the differential $\delta$ is given by counting gradient flow lines.}. More precisely, Witten introduced the twisted Laplacian 
$
\Delta_{f,\lam} := d^*_{f,\lam} \circ d +d \circ  d^*_{f,\lam}
$ \footnote{We let $d^*_{f,\lam}$ to be the adjoint of $d$, and $G_{f,\lam}$ to be Witten's Green function of $\Delta_{f,\lam}$ w.r.t.  volume form $e^{-2\lam f} \vol_M$.} with a large real parameter $\lam $, and an isomorphism
\begin{equation}\label{eqn:witten_map_introduction}
\phi: (CM^*_f,\delta) \rightarrow (\Omega^*_{f,<1}(M),d)
\end{equation} 
where $\Omega^*_{f,<1}(M)$ refers to the small eigensubspace of $\Delta_{f,\lam}$ (see Section \ref{sec:homological_perturbation}). The detailed analysis of $\phi$ is later carried out in \cite{HelSj1, HelSj2, HelSj3, HelSj4} and readers may also see \cite{zhang} for this correspondence.

In \cite{fukaya05}, Fukaya conjectured that Witten's isomorphism \eqref{eqn:witten_map_introduction} can be enhanced to an isomorphism of $A_\infty$ algebras (or categories), a generalization of differential graded algebras (abbrev. dga), encoding rational homotopy type by work of Quillen \cite{quillen1969} and Sullivan \cite{sullivan1977}. The $A_\infty$ structures $m_{k}(\lam)$'s on $\Omega^*_{f,<1}(M)$ are obtained by pulling back the structures of the de Rham dga $(\Omega^*(M),d,\wedge)$ using the homological perturbation lemma (see e.g. \cite{kontsevich00}) with homotopy operator $H_{f,\lam} = d^*_{f,\lam} G_{f,\lam}$. The Morse $A_\infty$ structures $m_k^{Morse}$'s are defined via counting gradient flow trees of Morse functions as in \cite{fukayamorse}. Fukaya conjectured that they are related by
\begin{equation}\label{eqn:introduction_fukaya_conjecture}
\lim_{\lam \rightarrow \infty} m_{k}(\lam)  = m_k^{Morse}
\end{equation}
via the Witten's isomorphism \eqref{eqn:witten_map_introduction}. This conjectured is proven in \cite{klchan-leung-ma} by extending the analytic technique in \cite{HelSj4} to incorporate the homotopy operator $H_{f,\lam}$.

When $M$ is equipped with a smooth $\mathbb{S}^1$ action, motivated by the geometry of loop space $\mathbb{S}^1 \curvearrowright \mathscr{L}X$ for some $X$, Getzler-Jones-Petrack \cite{getzler1991differential} introduced an enhancement of the equivariant de Rham complex on $M$. They defined new $A_\infty$ algebra structures consisting of 
\begin{equation}\label{eqn:introduction_equivariant_derham_structure}
\tilde{m}_k : \big( \Omega^*(M)[[u]] \big)^{\otimes k} \rightarrow \Omega^*(M)[[u]]
\end{equation} 
by adding higher order (in $u$) operations $u \eqopera_k$'s (see Section \ref{sec:equivariant_de_rham}) to ordinary de Rham dga structures. Witten's deformed $A_\infty$ structures $m_k(\lam)$'s are constructed from $\tilde{m}_k$'s in \eqref{eqn:introduction_equivariant_derham_structure} using the technique of homological perturbation as in original Fukaya's conjecture.

Inspired by Fukaya's correspondence, we define new Morse theoretic type counting structures $\morseprod{k}$'s (where $\morseprod{1}$ is known before in \cite{berghoff2012s}) associated to $\mathbb{S}^1 \curvearrowright M$, counting of Morse flow trees with jumpings coming from the $\mathbb{S}^1$ action (see the following Section \ref{sec:introduction_tree_with_jumping}). We prove the generalization of \eqref{eqn:introduction_fukaya_conjecture} for $\mathbb{S}^1 \curvearrowright M$ relating these two structures.

 \begin{theorem}[=Theorem \ref{thm:main_theorem}]\label{thm:introduction_theorem} We have
$$
 	\lim_{\lam \rightarrow \infty} m_{k}(\lam)  = \morseprod{k}.
 	$$
\end{theorem}

\subsection{The operation $\morseprod{k}$'s}\label{sec:introduction_tree_with_jumping}
To describe $\morseprod{k}$'s, we fix a generic sequence (see Definition \ref{def:moduli_space_as_intersection}) of functions $(f_0,\dots,f_k)$ such that their differences $f_{ij}: = f_j -f_i$ are assumed to be Morse-Smale as in Definition \ref{def:morsesmale}. The Morse theoretical $A_\infty$ product $\morseprod{k}$'s take the form
$$
\morseprod{k}:= \sum_{\tr}  \morseprod{k,\tr} : CM^*_{f_{(k-1)k}}[[u]] \otimes \cdots \otimes CM^*_{f_{01}}[[u]] \rightarrow CM^*_{f_{0k}}[[u]]
$$
 which is a summation over directed labeled ribbon $k$-tree $\tr$ with $k$-incoming edges and $1$ outgoing edge, where internal vertices are either labeled by $1$ or by $u$. For example (see Section \ref{sec:equivariant_morse} for details), if we take the tree $\tr$ to be the one with two incoming edges $e_{12}$ and $e_{01}$ joining the vertex $v_r$ connected to the outgoing edge $e_{02}$, with $v_r$ being labeled by $u$. The gradient flow trees with type $\tr$ will be consisting of gradient flow lines of $f_{12}$, $f_{01}$ and $f_{02}$ which ending at critical points $q_{12}$, $q_{01}$ and $q_{02}$ respectively, that can be joined together at a point $x_{v_r} \in M$ with further help of the $\mathbb{S}^1$ action $\actionmap_t : M \rightarrow M$ (for some $t$) as shown in the Figure \ref{fig:jumping_tree}. As a consequence of the above Theorem \ref{thm:introduction_theorem}, the Morse (pre)-category (here pre-category means this operation only defined for generic sequence $(f_0,\dots,f_k)$) on $\mathbb{S}^1 \curvearrowright M$ is an $A_\infty$ (pre)-category.

 \begin{figure}[h]
 	\centering
 	\includegraphics[scale = 0.65]{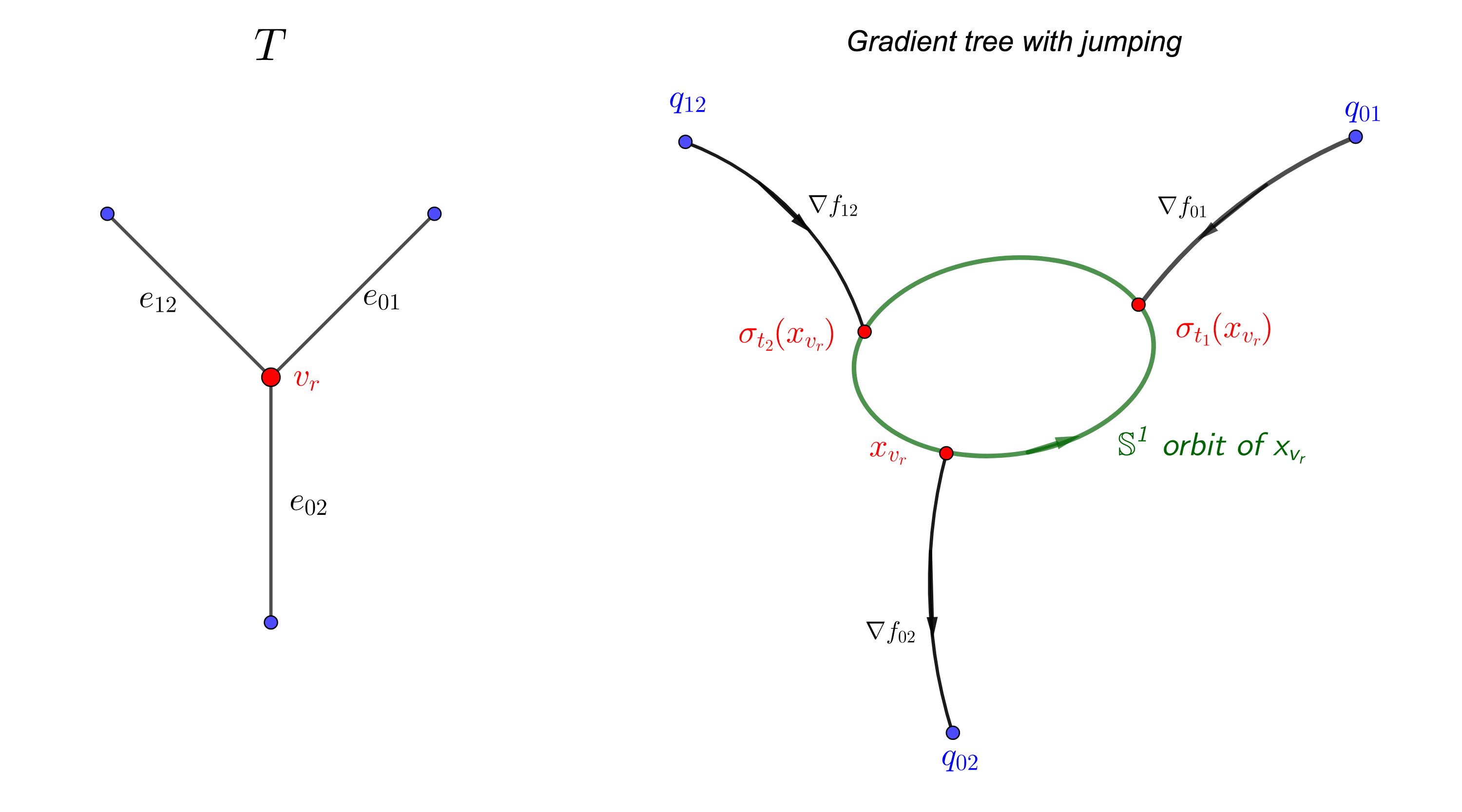}
 	\caption{Gradient tree with jumping of type $\tr$}\label{fig:jumping_tree}
 \end{figure}

\begin{corollary}
	The operations $\morseprod{k}$'s satisfy the $A_\infty$ relation for generic sequences of functions.
	\end{corollary}
 
 \begin{remark}
 	In \cite[Section 8b]{seidel2007biased}, Seidel proposed the $A_\infty$ operators $m_k^{\text{Floer}}$ on the symplectic cochain complex for a Liouville domain $X$, which corresponds to $\morseprod{k}$'s if we think of $M$ as a finite dimensional analogue of $\mathcal{L}X$. The corresponding $m_1^{\text{Floer}}$ operation is studied in details in \cite{zhao2014periodic}. The above Theorem \ref{thm:introduction_theorem} suggest how Witten deformation can provide a linkage between the Getzler-Jones-Petrack's operation $\tilde{m}_k$ on $\mathcal{L}X$ and the Floer theoretical operations introduced by Seidel through the investigation of the corresponding finite dimensional situation. 	
 \end{remark}

 This paper consists of three parts. In Section \ref{sec:witten_deformation} we set up the Witten deformation of Getzler-Jones-Petrack's $A_\infty$ operations $\tilde{m}_k$'s, the definition of counting gradient flow trees with jumping, and state our Main Theorem \ref{thm:main_theorem}. In Section \ref{sec:recalling_old_result}, we recall the necessary analytic result by following \cite{klchan-leung-ma}. The rest of Section \ref{sec:proof_of_theorem} will be a proof of Theorem \ref{thm:main_theorem} by figuring out the exact relations between the operations $m_{k,\tr}(\lam)$ and counting of gradient trees.
 
 \section*{Acknowledgement}
 The work in this paper is inspired by a talk of Naichung Conan Leung given at the Southern University of Science and Technology, and I would like to express my gratitude to Kwokwai Chan and Naichung Conan Leung for useful conversations when writing this paper.  
\section{Witten's deformation of $S^1$-equivariant de Rham complex}\label{sec:witten_deformation}
We always let $(M,g)$ to be an $n$-dimensional compact oriented Riemannian manifold, and denote it volume form by $\vol_M$ (or simply $\vol$). We assume there is an  smooth $\mathbb{S}^1$ action $\actionmap : \mathbb{S}^1 \times M \rightarrow M$ on $M$ preserving $(g,\vol)$. We should write $\actionmap_t : M \rightarrow M$ to be the action for a fixed $t \in \mathbb{S}^1$.

\subsection{$S^1$-equivariant de Rham complex and category}\label{sec:equivariant_de_rham}
We begin with recalling the Definition of $S^1$-equivariant de Rham $A_\infty$ algebra introduced in \cite{getzler1991differential}, which is reformulated to be $A_\infty$ category as follows for the convenient of presentation of this paper.

\begin{definition}\label{def:s_1_equivariant_de_rham}
	The $S^1$-equivariant de Rham $A_\infty$ category $\dr{}$ consisting of object being smooth functions $f : M \rightarrow \real$, with morphism $\Hom(f,g) := \Omega^*(M)[[u]]$ where $u$ is a formal variable. The $A_\infty$ operations $\tilde{m}_k : \Hom(f_{k-1},f_k) \otimes \cdots \otimes \Hom(f_0,f_1) \cong (\Omega^*(M)[[u]])^{\otimes k} \rightarrow \Hom(f_0,f_k) \cong \Omega^*(M)[[u]]$ is defined by $\tilde{m}_1(\alpha_{01}) = d(\alpha_{01}) + u \eqopera_1(\alpha_{01})$, $\tilde{m}_2(\alpha_{12},\alpha_{01}) = (-1)^{|\alpha_{12}|+1} \alpha_{12}\wedge \alpha_{01} + u \eqopera_2(\alpha_{12},\alpha_{01})$ and $\tilde{m}_k(\alpha_{(k-1)k},\dots,\alpha_{01}) = u \eqopera_k(\alpha_{(k-1)k},\dots,\alpha_{01})$ for $\alpha_{ij} \in \Hom (f_i,f_j)$. 
	
	Here the operator $\eqopera_k$ is defined by the action $\eqopera_1(\alpha_{ij}) = \int_{\mathbb{S}^1} (\iota_{\dd{t}}\actionmap^*(\alpha_{ij})) dt$, and for $k\geq2$ we use
	$$
	\eqopera_k(\alpha_{(k-1)k},\dots,\alpha_{01}):= \int_{0\leq t_k \leq \cdots \leq t_1 \leq 1} \left( \iota_{\dd{t_k}}(\actionmap^*(\alpha_{(k-1)k})) \wedge \cdots \wedge  \iota_{\dd{t_1}}(\actionmap^*(\alpha_{01}))  \right) dt_k \cdots dt_1.
	$$ 
	\end{definition}
The fact that the about operations $\tilde{m}_k$'s form an $A_\infty$ category is proven in \cite[Theorem 1.7]{getzler1991differential}. 

\subsection{Homological perturbation via Witten's deformation}\label{sec:homological_perturbation}
We follow \cite[Section 2.2.]{klchan-leung-ma} to introduced the Witten deformation with a real parameter $\lam >0$, which is orignated from \cite{witten82}. For each $f_i$ and $f_j$, we twist the volume form $\vol$ by $f_{ij}:= f_j - f_i$ as $\vol_{ij} = e^{- 2\lam f_{ij}} \vol$, and let $d_{ij}^*:= e^{2\lam f_{ij}} d^* e^{-2\lam f_{ij}} = d^* + 2 \lam \iota_{\nabla f_{ij}}$ to be the adjoint of $d$ with respect to the volume form $\vol_{ij}$. The Witten Laplacian is defined by 
$
\Delta_{ij}:= d d^*_{ij} + d^*_{ij} d,
$
acting on the complex $\Omega^*(M)[[u]]$ \footnote{Stictly speaking, the differential forms here depend on the real parameter $\lam$ while we prefer to subpress the dependence in our notation.}. We denote the span of eigenspaces with eigenvalues contained in $[0,1)$ by $\Omega^*_{ij,<1}(M)[[u]]$, or simply $\Omega^*_{ij,<1}[[u]]$. We use construction in \cite{klchan-leung-ma} originated from \cite{fukaya05} using homological perturbation lemma \cite{kontsevich00}, which obtain a new $A_\infty$ structure from $m_k$'s  as follows.

\begin{definition}\label{def:labeled_k_tree}
	A {\em (directed) $k$-tree} labeled $\tr$ consists of a finite set of vertices $\bar{\tr}^{[0]}$ together with a decomposition
	$\bar{\tr}^{[0]} = \tr^{[0]}_{in} \sqcup \tr^{[0]} \sqcup \{v_o\},$
	where $\tr^{[0]}_{in}$, called the set of incoming vertices, is a set of size $k$ and $v_o$ is called the outgoing vertex (we also write $\tr^{[0]}_\infty := \tr^{[0]}_{in} \sqcup \{v_o\}$ and $\tr^{[0]}_{ni} := \tr^{[0]} \cup \{v_o\}$), a finite set of edges $\bar{\tr}^{[1]}$, two boundary maps $\partial_{in} , \partial_o : \bar{\tr}^{[1]} \rightarrow \bar{\tr}^{[0]}$ (here $\partial_{in}$ stands for incoming and $\partial_o$ stands for outgoing), and a labeling of every internal vertices $\tr^{[0]}$ by either $1$ or $u$, satisfying the following conditions:
	\begin{enumerate}
		\item
		Every vertex $v \in \tr^{[0]}_{in}$ has valency one, and satisfies $\# \partial_{o}^{-1}(v) = 0$ and $\# \partial_{in}^{-1}(v) = 1$; we let $\tr^{[1]} := \bar{\tr}^{[1]}\setminus \partial_{in}^{-1}(\tr^{[0]}_{in})$.
		\item Every vertex $v \in \tr^{[0]}$ has an unique edge $e_{v,o} \in \bar{\tr}^{[1]}$ such that $\partial_{in}(e_{v,o}) = v$, and only trivalent vertices in $\tr^{[0]}$ can be labeled with $1$. 
		\item
		For the outgoing vertex $v_o$, we have $\# \partial_{o}^{-1}(v_o) = 1$ and $\# \partial_{in}^{-1}(v_o) = 0$; we let $e_o := \partial_o^{-1}(v_o)$ be the outgoing edge and denote by $v_r \in \tr^{[0]}_{in} \sqcup \tr^{[0]}$ the unique vertex (which we call the root vertex) with $e_o = \partial^{-1}_{in}(v_r)$.
		\item
		The {\em topological realization}
		$|\bar{\tr}| := \left( \coprod_{e \in \bar{\tr}^{[1]}} [0,1] \right) / \sim$
		of the tree $\tr$ is connected and simply connected; here $\sim$ is the equivalence relation defined by identifying boundary points of edges if their images in $\tr^{[0]}$ are the same.
	\end{enumerate}
	By convention we also allow the unique labeled $1$-tree with $\tr^{[0]} = \emptyset$. Two labeled $k$-trees $\tr_1$ and $\tr_2$ are {\em isomorphic} if there are bijections $\bar{\tr}^{[0]}_1 \cong \bar{\tr}^{[0]}_2$ and $\bar{\tr}^{[1]}_1 \cong \bar{\tr}^{[1]}_2$ preserving the decomposition $\bar{\tr}^{[0]}_i = \tr^{[0]}_{i,in} \sqcup \tr^{[0]}_i \sqcup \{v_{i,o}\}$ and boundary maps $\partial_{i,in}$ and $\partial_{i,o}$ and the labelling of $\tr^{[0]}$. The set of isomorphism classes of labeled $k$-trees will be denoted by $\tree{k}$. For a labeled $k$-tree $\tr$, we will abuse notations and use $\tr$ (instead of $[\tr]$) to denote its isomorphism class.
	
	A labeled ribbon $k$-tree is a $k$-tree $\tr$ with a cyclic ordering of $\partial_{in}^{-1}(v) \sqcup \partial_o^{-1}(v)$ for each trivalent vertex $v \in \tr^{[0]}$, and isomorphism of labeled ribbon $k$-trees are further required to preserve this ordering. A labeled ribbon $k$-tree can have its topological realization $|\bar{\tr}|$ being embedded into the unit disc $D$, with $\tr^{[0]}_\infty$ lying on the boundary $\partial D$ such that the cyclic ordering of $\partial_{in}^{-1}(v) \sqcup \partial_o^{-1}(v)$ agree with the anti-clockwise orientation of $D$. The set of isomorphism classes of labeled ribbon $k$-trees will be denoted by $\ltree{k}$.
	\end{definition}

\begin{notation}\label{not:morse_function_associated_to_edges}
	For each $\tr \in \ltree{k}$, we can associated to each edge $e \in \bar{\tr}^{[1]}$ a numbering by pair of integer $ij$ using the embedding $|\bar{\tr}| \rightarrow D$ by the rules: there are $k+1$ connected components of $D \setminus |\bar{\tr}| $, and we assign each component by integers $0,\dots,k$; each (directed) edge $e \in \bar{\tr}^{[1]}$ with region numbered by $i$ on its left and region numbered by $j$ on its right is numbered by $ij$; the incoming edges numbered by $e_{(k-1)k},\dots ,e_{01}$ and the outgoing edge $e_{0k}$ are in clockwise ordering of $\partial D$. 
	
	A pair of $v \in \tr^{[0]} \cup \{v_o\}$ attached to an edge $e \in \bar{\tr}^{[1]}$ is called a flag, and we will let $\digamma(\tr)$ to be the set of all flags. For every flag $(e,v)$, we let $\tr_{e,v}$ to be the unique subtree with outgoing vertex being $v$ if $\partial_o(e) = v$, and we let $\tr_{e,v}$ to be the unique subtree with outgoing edge being $e$ if $\partial_{in}(e) = v$.
	\end{notation}

\begin{definition}\label{def:operation_associated_to_labeled_tree}
	Given a labeled ribbon $k$-tree $\tr \in \ltree{k}$ with an embedding $|\bar{\tr}| \rightarrow D$, we assoicate to it an operation $m_{k,\tr}(\lam) : \Omega^*_{(k-1)k,<1}[[u]] \otimes \cdots \otimes \Omega^*_{01,<1}[[u]] \rightarrow \Omega^*_{0k,<1}[[u]]$ by the following rules :
\begin{enumerate}
	\item aligning the inputs $\varphi_{(k-1)k},\cdots,\varphi_{01}$ at the incoming vertices $\tr^{[0]}_{in}$ according to the clockwise ordering induced from $D$;
	\item  if a vertex $v \in \tr^{[0]}$ has incoming edges $e_{v,1},\dots,e_{v,l}$ and outgoing edge $e_{v,o}$ attached to it such that $e_{v,l},\dots,e_{v,1}, e_{v,o}$ is in clockwise orientation, we apply the operation $\wedge$ if $v$ is labeled with $1$ (and hence trivalent) and the operation $\eqopera_l$ if $v$ is labeled with $u$;
	\item for an edge $e \in \tr^{[0]}$ which is numbered by $ij$, we apply the homotopy operator $H_{ij}:= d^*_{ij} G_{ij}$ where $G_{ij}$ is the Witten's twisted Green operator associated to the Witten Laplacian $\Delta_{ij}$;
	\item for the unique outgoing edge $e_o$, we apply the operator $P_{0k}$ which is the orthogonal projection $P_{0k} : \Omega^*[[u]] \rightarrow \Omega^*_{0k,<1}[[u]]$ with respect to the twisted $L_2$-norm obtained from the volume form $\vol_{0k}$. 
\end{enumerate}
By convention, we define $m_{1,\tr}(\lam)$ for the unique tree with $\tr^{[0]} = \emptyset$ to be the restriction of $d$ on $\Omega^*_{ij,<1}[[u]]$. For each labeled ribbon $k$-tree $\tr$, we assign $n_{\tr}$ to be the number of vertices in $\tr^{[0]}$ labeled with $u$, and we let $m_k(\lam):= \sum_{\tr \in \ltree{k}} u^{n_{\tr}} m_{k,\tr}(\lam)$ to be the homological perturbed $A_\infty$ strucutre. 
	\end{definition}

It is well-known that (see e.g. \cite[Chapter 8]{dbrane}) the perturbed $A_\infty$ structure $m_k(\lam)$'s satisfy the $A_\infty$ relation. And we obtain a new category $\dr{<1}$ via Witten deformation. 

\subsection{Relation with $S^1$-equivariant Morse flow trees}\label{sec:equivariant_morse}
In \cite{ HelSj4,witten82,zhang}, a relation between the Morse complex $\mcpx{f_{ij}}$ and $\Omega^*_{ij,<1}$ is established when $f_{ij}$ is a Morse-Smale function in following Definition \ref{def:morsesmale}. Following \cite{zhang}, it is an isomorphism 
\begin{equation}\label{eqn:witten_map}
\Phi_{ij} : \Omega^*_{ij,<1} \rightarrow \mcpx{f_{ij}} ; \quad \Phi_{ij}(\alpha):= \sum_{p \in \text{Crit}(f_{ij})} \int_{V_p^-} \alpha,
\end{equation}
where $\text{Crit}(f_{ij})$ is the finite set of critical points of $f_{ij}$ (with Morse index of $p$ given by number of negative eigenvalues of $\nabla^2 f_{ij}(p)$), and $V_p^-$ (Notice that we further choose an orientation of $V_p^-$ by choosing a volume element of the normal bundle $NV_p^+$) is the unstable submanifold associated to $p$ which is the union of all gradient flow lines $\gamma(s)$ of $\nabla f_{ij}$ which limit toward $p$ as $s \rightarrow \infty$. Furthermore, the de Rham differential is identified with the Morse differential $\delta_1$ defined via counting Morse flow lines. 
\begin{definition}\label{def:morsesmale}
	A Morse function $f_{ij}$ is said to satisfy the Morse-Smale condition if $V^+_p$ and $V^-_q$ intersecting transversally for any two critical points $p\neq q$ of $f_{ij}$.
\end{definition}
We illustrate how the technique in \cite{klchan-leung-ma} can be used to establish a relation between $\lam\rightarrow \infty$ limit of the operation $m_k^{\tr}(\lam)$ with a new Morse-theoretical counting for $\mathbb{S}^1 \rightarrow M$ defined as follows. 

\begin{notation}\label{not:metric_k_tree}
	A metric labeled $k$-tree (ribbon) $\mtr$ is a labeled (ribbon) $k$-tree together with a length function $l:\tr^{[1]} \setminus \{e_o\}\rightarrow(0,+\infty)$. For each $e \in \bar{\tr}^{[1]}$, we let $\mathcal{I}_e = (-\infty,0]$ if $e \in \tr^{[1]}_{in}$, $\mathcal{I}_e = [0, l(e)]$ for $e \in \tr^{[1]} \setminus \{e_o\}$ and $\mathcal{I}_{e_o} = [0,\infty)$. The space of metric structure on $\tr$, denoted by $\mathcal{S}(\tr)$, is a copy of $(0,+\infty)^{|\tr^{[1]}|-1}$. The space $\mathcal{S}(\tr)$ can be partially compactified to a manifold with corners $(0,+\infty]^{|\tr^{[1]}|-1}$, by allowing the length of internal edges going to be infinity. In particular, it has codimension-$1$ boundary $
	\partial \overline{\mathcal{S}(\tr)} = \coprod_{\tr=\tr{'}\sqcup \tr{''}} \mathcal{S}(\tr{'})\times \mathcal{S}(\tr{''})$.
	
	For every vertex $v \in \bar{\tr}$, we use $\val(v)+1$ to denote the valency of $v$. We write  $\simplex_l := \{ (t_l,\dots,t_1) \in [0,1]^{l} \ | \ 0\leq t_l \leq \cdots \leq t_1 \leq 1 \}$ for $l>1$, and $\simplex_1 = \mathbb{S}^1$ \footnote{This is not the $1$-simplex, but we would like to unify our notation in this way.}, and attach to each vertex $v$ labeled with $u$ a simplex $\simplex_{\val(v)}$. Writing $\lvertex$ to be the collection of all vertices with label $u$, we let $\mathbf{S}(\tr):= \prod_{v \in \lvertex} \simplex_{\val(v)} \times \mathcal{S}(\tr)$. 
	\end{notation}

\begin{definition}\label{def:equivariant_morse_counting}
	Given a sequence $\vec{f} = (f_0,\dots,f_k)$ such that all the difference $f_{ij}$'s are Morse, with a sequence of points $\vec{q} = (q_{(k-1)k},\dots ,q_{01},q_{0k})$ such that $q_{ij}$ is a critical point of $f_{ij}$, and a metric labeled ribbon $k$-tree $\mtr$, a gradient flow tree (with jumping) $\ftree$ (readers may see Figure \ref{fig:jumping_tree} for an example) of type $(\tr,\vec{f},\vec{q})$ consisting of a gradient flow line $\gamma_{ij} : \mathcal{I}_{e_{ij}} \rightarrow M$ of the Morse function $f_{ij}$ for each edge $e_{ij} \in \bar{\tr}^{[1]}$ numbered by $ij$, and a point $\mathbf{t}_{v} = (t_{v,\val(v)},\dots,t_{v,1}) \in \simplex_{\val(v)}$ for every $v \in \lvertex$ satisfying:
	\begin{enumerate}
		\item $\lim_{s \rightarrow -\infty} \gamma_{e_{i(i+1)}}(s) = q_{i(i+1)}$ for the incoming edges $e_{i(i+1)} \in \tr^{[1]}_{in}$, and $\lim_{s \rightarrow \infty} \gamma_{e_{0k}}(s) = q_{0k}$ for the unique outgoing edge $e_o$;
		\item for a trivalent vertex $v \in \tr^{[0]}$ labeled by $1$ with two incoming edges $e_{jl}$, $e_{ij}$ and outgoing edge $e_{il}$, we require that $\gamma_{ij}(l(e_{ij})) = \gamma_{jl}(l(e_{jl})) = \gamma_{il}(0)$;
		\item for a vertex $v \in \lvertex$ with incoming edges $e_{i_{l-1} i_{l}}, \dots, e_{i_0 i_1}$ and outgoing edge $e_{i_0 i_{l}}$, we require that $\actionmap( -t_{v,l}, \gamma_{i_{l-1} i_{l}}(l(e_{i_{l-1} i_{l}})))  = \cdots  = \actionmap (-t_{v,1}, \gamma_{i_0 i_1}(l(e_{i_0 i_1}))) = \gamma_{i_0 i_{l}} (0)$, where $l = \val(v)$ and $\actionmap$ is the $\mathbb{S}^1$ action map in the beginning of Section \ref{sec:witten_deformation}. 
		\end{enumerate}
	We will let $\mathcal{M}_{\tr}(\vec{f},\vec{q})$ to denote the moduli space (as a set) of gradient flow lines of type $\tr$. For the unique tree with $\tr^{[0]} = \emptyset$, we let $\mathcal{M}_\tr(\vec{f},\vec{q})$ to be the moduli space of gradient flow lines quotient by the extra $\real$ symmetry by convention.
	\end{definition}

Similar to the moduli space of gradient flow trees without $\mathbb{S}^1$ action (see e.g. \cite[Section 2.1.]{klchan-leung-ma}), we can describe $\mathcal{M}_{\tr}(\vec{f},\vec{q})$ as intersection of stable and unstable submanifolds. 

\begin{definition}\label{def:moduli_space_as_intersection}
	Given the sequence $\vec{f}$ and $\vec{q}$ as in the above Definition \ref{def:equivariant_morse_counting}, we define a smooth map $\mathbf{f}_{\tr,i(i+1)} : V_{q_{i(i+1)}}^+ \times \mathbf{S}(\tr) \rightarrow  M$ for each $i = 0,\dots,k-1$ as follows. Given a incoming edge $e_{i(i+1)}$, there is a unique sequence of edges $e_{i_0 j_0}= e_{i(i+1)} , e_{i_1 j_1} , \dots , e_{i_m j_m}, e_{i_{m+1} j_{m+1}} = e_{o}$ with $v_d := \partial_o(e_{i_d j_d})$ forming a path from the incoming vertex $v_{i(i+1)}$ to the outgoing vertex $v_o$. Fixing a point $x_0 \in V_{q_{i(i+1)}}^+$ and a point $((\mathbf{t}_v)_{v \in \lvertex}, (l(e))_{e \in \tr^{[1]} \setminus \{e_o\}}) \in \mathbf{S}(\tr)$, we determind a point $x_d \in M$ inductively for $ 0 \leq d \leq m+1$ by the rules:
	\begin{enumerate}
		\item  if $v_d$ is labeled with $1$, we simply take $x_{d+1}$ to be the image of $x_d$ under $l(e_{i_{d+1} j_{d+1}})$ time flow of $\nabla f_{i_{d+1}j_{d+1}}$ for $d<m$, and $x_{d+1} = x_{d}$ for $d=m$; 
		\item and if $v_d$ is labeled with $u$, we take $x_{d+1}$ to be the image of $\actionmap(-t_{v_d, l}, x_d)$ under the $l(e_{i_{d+1} j_{d+1}})$ time flow of $\nabla f_{i_{d+1}j_{d+1}}$ if $d<m$, and $x_{d+1}=\actionmap(-t_{v_d, l}, x_d)$ for $d=m$, where $e_{i_dj_d}$ is the $l$-th incoming edge attached to $v_d$ in the {\em anti-clockwise} orientation.
		\end{enumerate}
	
	These map can be put together as $\mathbf{f}_\tr : V_{q_{0k}}^- \times V_{q_{(k-1)k}}^+ \times \cdots \times V_{q_{01}}^+ \times \mathbf{S}(\tr) \rightarrow M^{k}$ using the natural embedding $V_{q_{0k}}^- \hookrightarrow M$ for the first component. Therefore we see that $\mathcal{M}_\tr(\vec{f},\vec{q}) = \mathbf{f}_\tr^{-1}(\mathbf{D})$ where $\mathbf{D} = M \hookrightarrow M^{k+1}$ is the diagonal.  
	
	We say a sequence of function $\vec{f}$ generic if for any sequence of critical points $\vec{q}$, any labeled tree $\tr$ the associated intersection $\mathbf{f}_\tr$ with $\mathbf{D}$ is transversal with expected dimension (meaning that it is empty when expected negative dimensional intersection), and the same hold when restricting $\mathbf{f}_\tr$ on any boundary strata of $V_{q_{0k}}^- \times V_{q_{(k-1)k}}^+ \times \cdots \times V_{q_{01}}^+ \times \mathbf{S}(\tr)$ (the stratification coming from that of $\simplex_{\val(v)}$) and for any subsequence of $\vec{f}$. 
	\end{definition}

Suppose we are given a generic sequence $\vec{f}$ with $\vec{q}$ and $\tr$ as in the above Definition \ref{def:moduli_space_as_intersection}, then we can compute the dimension of the moduli space as
\begin{equation}\label{eqn:dimension_of_moduli_space}
\dim(\mathcal{M}_\tr(\vec{f},\vec{q})) = \deg(q_{0k}) - \sum_{i=0}^{k-1} \deg(q_{i(i+1)}) + \sum_{v \in \lvertex} \val(v) + |\tr^{[1]} | - 1. 
\end{equation}

\begin{definition}\label{def:sign_of_morse_counting}
	Given generic $\vec{f}$, $\vec{q}$ and $\tr$ as in the above Definition \ref{def:moduli_space_as_intersection} such that $\dim(\mathcal{M}_\tr(\vec{f},\vec{q}))=0$, with a flow tree $\ftree \in \mathcal{M}_\tr(\vec{f},\vec{q})$, we assign a sign $(-1)^{\chi(\ftree)}$ by assigning a differential form $\vol_{e,v} \in \bigwedge^n T^*M_{\gamma_{e}(v)}$ (Here we abuse the notation to use $v$ to stand for the corresponding point in $\mathcal{I}_{e}$) for each flag $(e,v) \in \digamma(\tr)$, inductively along the tree $\tr$ as follows: 
	\begin{enumerate}
		\item for an incoming edge $e_{i(i+1)}$ with $v = \partial_o(e_{i(i+1)})$, we let $\vol_{e_{i(i+1)},v}$ to be the restriction of the volume form of the normal bundle $NV_{q_{i(i+1)}}^+$ onto $\gamma_{e_{i(i+1)}}(v)$;
		\item for a vertex $v \in \tr^{[0]}$ with incoming edges $e_{i_{l-1} i_{l}}, \dots, e_{i_0 i_1}$ and outgoing edge $e_{i_0 i_{l}}$ arranged in clockwise orientation with $\vol_{e_{i_{d-1} i_d},v}$ defined, we let $\vol_{e_{i_0i_2},v}:=(-1)^{|\vol_{e_{i_{2}i_1},v}|+1} \vol_{e_{i_{2}i_1},v} \wedge \vol_{e_{i_0 i_1},v}$ when $v$ is labeled with $1$ \footnote{Hence we have valency of $v$ being $3$.}, and we let $\vol_{e_{i_0 i_{l}},v}:= \actionmap_{t_{v,l}}^*( \iota_{\actionmap_*(\dd{t_l})} \vol_{e_{i_{l-1} i_{l}},v}) \wedge \cdots \wedge \actionmap_{t_{v,1}}^* (\iota_{\actionmap_*(\dd{t_1})} \vol_{e_{i_0 i_1},v})$ when $v$ is labeled with $u$;
		\item for an edge $e_{ij}$ with incoming vertex $v_0 = \partial_{in}(e_{ij})$ and outgoing vertex $v_1 = \partial_o(e_{ij})$, we let $\vol_{e_{ij},v_1} = (\tau_{l(e_{ij})})_* (\vol_{e_{ij},v_0})$ where $\tau_{l(e_{ij})}$ is the gradient flow of $\nabla f_{ij}$ for time $l(e_{ij})$. 
	\end{enumerate}
Therefore, for the outgoing edge $e_{0k}$ starting at the root vertex $v_r$ and ending at the outgoing vertex $v_o$, we obtain a differential form $\vol_{e_{0k},v_r}$ from the above construction, and we determine the sign $(-1)^{\chi(\ftree)}$ by $(-1)^{\chi(\ftree)} \vol_{e_{0k},v_r} \wedge \ast \vol_{q_{0k}} = \vol_{M}$ where $\vol_{q_{0k}}$ is the chosen volume element in $NV_{q_{0k}}^+$ for the critical point $q_{0k}$. (For the case $\tr^{[0]} = \emptyset$, we define by convention that $(-1)^{\chi(\ftree)} \ftree' \wedge \vol_{p} \wedge \ast \vol_{q} = \vol_M$ for a gradient flow line $\ftree$ from $p$ to $q$.)
	\end{definition}

\begin{definition}\label{def:morse_A_infinity_structure}
	Given a generic sequence of functions $\vec{f} = (f_0,\dots,f_k)$, with a sequence of critical points $(q_{(k-1)k},\dots,q_{01})$ we define the operation $\morseprod{k}(q_{(k-1)k},\dots,q_{01}) \in \mcpx{f_{0k}}^*[[u]]$ by extending linearly the formula
	$$
	\morseprod{k,\tr}(q_{(k-1)k},\dots,q_{01}) := \begin{dcases} \sum_{q_{0k} \in \text{Crit}(f_{0k})} \big( \sum_{\ftree \in \mathcal{M}_\tr(\vec{f},\vec{q})}  (-1)^{\chi(\ftree)} \big) q_{0k} & \text{if $\dim(\mathcal{M}_\tr(\vec{f},\vec{q})) = 0$,} \\
	0 & \text{otherwise,}
	\end{dcases}
	$$
	where $\vec{q} = ((q_{(k-1)k},\dots,q_{01},q_{0k})$. We further let $\morseprod{k} = \sum_{\tr \in \ltree{k}} u^{n_{\tr}} \morseprod{k,\tr}$ where $n_\tr = |\lvertex|$. 
	\end{definition}

We have the following Theorem \ref{thm:main_theorem} which is the main result for this paper.
\begin{theorem}\label{thm:main_theorem}
	Given a generic sequence of functions $\vec{f} = (f_0,\dots,f_k)$, with a sequence of critical points $\vec{q} = (q_{(k-1)k},\dots,q_{01},q_{0k})$, then we have
	$$
	\lim_{\lam \rightarrow \infty} \Phi \big( m_{k,\tr}(\lam)(\phi(q_{(k-1)k}),\dots,\phi(q_{01}) ) \big) = \morseprod{k,\tr}(q_{(k-1)k},\dots,q_{01}),
	$$
	where $\phi:= \Phi^{-1}$ \footnote{We omit the numbering $ij$ from our notation here.} is the inverse of the isomorphism in equation \eqref{eqn:witten_map}.
	
	As a consequence, the Morse product $\morseprod{k}$'s satisfy the $A_\infty$-relation whenever we consider a generic sequence of functions such that every operation appearing in the formula is well-defined.
	\end{theorem}

\section{Proof of Theorem \ref{thm:main_theorem}}\label{sec:proof_of_theorem}

\subsection{Analytic results}\label{sec:recalling_old_result}
For the proof of Theorem \ref{thm:main_theorem}, we assume $\tr^{[0]} \neq \emptyset$ since this is exactly the case carried out by \cite{HelSj4}. We begin with recalling the necessary analytic results from \cite{HelSj4, zhang, klchan-leung-ma}. 

\subsubsection{Results for a single Morse function}
We will assume that the function $f_{ij}$ we are dealing with satisfy the Morse-Smale assumption \ref{def:morsesmale}. Due to difference in convention, $e^{-\lam f_{ij}} \Delta_{ij} e^{\lam f_{ij}}$ is called the Witten's Laplacian in \cite{klchan-leung-ma}, and result stated in this Section is obtain by the corresponding statements in \cite{klchan-leung-ma} by conjugating $e^{\lam f_{ij}}$.

\begin{theorem}[\cite{HelSj4, zhang}]\label{thm:witten_map_iso}
	For each $f_{ij}$, there is $\lam_0>0$ and constants $c, C>0$ such that we have
	$
	\Spec(\Delta_{ij})\cap [ce^{-c\lam}, C\lam^{1/2}) = \emptyset,$
	for $\lam>\lam_0$. The map $\Phi = \Phi_{ij} : \Omega^*_{ij,<1} \rightarrow \mcpx{f_{ij}}^*$ in equation \eqref{eqn:witten_map} is a chain isomorphism for $\lam$ large enough. We will denote the inverse by $\phi = \phi_{ij}$.
\end{theorem}

We will the asymptotic behaviour of $\phi(q)$ for a critical point $q$ of $f_{ij}$, and we will need the following Agmon distance $\dist_{ij}$ for this purpose.

\begin{definition}\label{def:agmondistance}
	For a Morse function $f_{ij}$, the Agmon distance $\dist_{ij}$ \footnote{Readers may see \cite{HelNi} for its basic properties.}, or simply denoted by $\dist$, is the distance function with respect to the degenerated Riemannian metric $\langle \cdot, \cdot \rangle_{f_{ij}} = |df_{ij}|^2 \langle \cdot, \cdot \rangle$, where $\langle \cdot, \cdot \rangle$ is the background metric. We will also write $\tdist_{ij}(x,y) := \dist_{ij}(x,y) - f_{ij}(y) + f_{ij}(x)$. 
\end{definition}

\begin{lemma}\label{lem:agmon_dist_flow_line}
	We have $\tdist_{ij}(x,y) \geq 0$ with equality holds if and only if $x$ is connected to $y$ via a generalized flow line $\gamma : [0,1] \rightarrow M$ with $\gamma(0) = x$ and $\gamma(1) = y$. Here a generalized flow line means that $\gamma$ is continuous, and there is a partition $0=t_0 <t_1< \cdots <t_l = 1$ such that $\gamma|_{(t_r,t_{r+1})}$ is a reparameterization of a gradient flow line of $f_{ij}$ and $\gamma(t_r) \in Crit(f_{ij})$ for $0<r<l$. 
\end{lemma}

\begin{lemma}\label{lem:resolventestimate}
	Let $\gamma \subset \comp$ to be a subset whose distance from $Spec(\Delta_{ij})$ is bounded below by a constant. For any $j\in \inte_+$ and $\epsilon>0$, there is $k_j \in \inte_+$ and $\lam_0 = \lam_0(\epsilon)>0$ such that for any two points $x_0, y_0 \in M$, there exist neighborhoods $V$ and $U$ (depending on $\epsilon$) of $x_0$ and $y_0$ respectively, and $C_{j,\epsilon}>0$ such that
	$
	\| \nabla^j ((z-\Delta_{ij})^{-1}u)\|_{C^{0}(V)} \leq C_{j,\epsilon} e^{-\lam(\tdist_{ij}(x_0,y_0)-\epsilon)} \| u\|_{W^{k_j,2}(U)},
	$
	for all $\lam>\lam_0$ and $u \in C^0_c(U)$, where $W^{k,p}$ refers to the Sobolev norm.
\end{lemma}

We will also need modified version of the resolvent estimate for $G_{ij}$, which can be obtained by applying the original resolvent estimate to the the formula
\begin{equation}
G_{ij}(u) = \oint_{\gamma} z^{-1}(z-\Delta_{ij})^{-1} u.
\end{equation}

\begin{lemma}\label{lem:resolventlemma}
	For any $j\in \inte_+$ and $\epsilon>0$, there is $k_j \in \inte_+$ and $\lam_0 = \lam_0(\epsilon)>0$ such that for any two points $x_0, y_0 \in M$, there exist neighborhoods $V$ and $U$ (depending on $\epsilon$) of $x_0$ and $y_0$ respectively, and $C_{j,\epsilon}>0$ such that $
	\| \nabla^j (G_{ij}u)\|_{C^{0}(V)} \leq C_{j,\epsilon} e^{-\lam(\tdist_{ij}(x_0,y_0)-\epsilon)} \| u\|_{W^{k_j,2}(U)}$,
	for all $\lam<\lam_0$ and $u \in C^0_c(U)$, where $W^{k,p}$ refers to the Sobolev norm.
\end{lemma}

For a critical point $q$ of $f_{ij}$, $\phi(q)$, has certain exponential decay measured by the Agmon distance from the critical point $q$.

\begin{lemma}\label{lem:eigenestimate2}
	For any $\epsilon$, there exists $\lam_0=\lam_0(\epsilon)>0$ such that for $\lam>\lam_0$, we have
	$
	\phi(q) = \mathcal{O}_{\epsilon}(e^{-\lam(\g^+_q(x)-\epsilon)}),
	$
	and same estimate holds for the derivatives of $\phi_{ij}(q)$ as well. Here $\mathcal{O}_\epsilon$ refers to the dependence of the constant on $\epsilon$ and $\g^+_q(x) =\tdist_{ij}(q,x) =  \dist_{ij}(q,x)+f_{ij}(q) - f_{ij}(x)$.
\end{lemma}

\begin{remark}\label{rem:eigenwkbremark}
	We notice that $\g^+_{q} $ is a nonnegative function with zero set $V_{q}^+$ that is smooth and Bott-Morse in a neighborhood $W$ of $V_{q}^+ \cup V_q^-$. Similarly, if we write $\g^-_{q} = \dist_{ij}(q,x) + f_{ij}(x) - f_{ij}(q)$ which is a nonnegative function with zero set $V_{q}^-$ and is smooth and Bott-Morse in $W$, and we have $\ast_{ij} \phi(q)/\|\phi(q) e^{-\lam f_{ij}}\|^2 =   \mathcal{O}_{\epsilon}(e^{-\lam(\g^-_{q}-\epsilon)})$ where $\ast_{ij} =\ast e^{-2\lam f_{ij}}$ comparing to the usual star operator $\ast$.
	\end{remark}

\begin{lemma}
	The normalized basis $\phi(q)/\|\phi(q)\|$'s are almost orthonormal basis with respect to the twisted inner product $\langle \cdot,\cdot \rangle e^{-2 \lam f_{ij}}$. More precisely, there is a $C,c>0$ and $\lam_0$ such that when $\lam > \lam_0$, we will have
	$
	\int_M \langle \frac{\phi(p)}{\|\phi(p)\|} , \frac{\phi(q)}{\|\phi(q)\|} \rangle \vol_{ij} = \delta_{pq} + Ce^{-c\lam}.
	$
\end{lemma}

Restricting our attention to a small enough neighborhood $W$ containing $V^+_{q} \cup V^-_{q}$, the above decay estimate of  $\phi(q)$ from \cite{HelSj4} can be improved from an error of order $\mathcal{O}_\epsilon(e^{\epsilon \lam})$ to $\mathcal{O}(\lam^{-\infty})$. 
\begin{lemma}\label{lem:eigenwkb}
	There is a WKB approximation of the $\phi(q)$ as 
	$
	\phi(q) \sim \lam^{\frac{\deg(q)}{2}} e^{-\lam \g^+_{q}}( \omega_{q,0}+\omega_{q,1}\lam^{-1/2}+\dots) \footnote{Notice that we indeed have $\omega_{q,2j+1} = 0$ in this case while we prefer to write it in this form to unify our notations.},
	$
	which is an approximation in any precompact open subset $K\subset W_q$ of the form
	$$
	\|e^{ \lam \g^+_q} \nabla^j \big( \lam^{-\deg(q)/2} \phi(q) - e^{-\lam \g^+_q }\sum_{l=0}^{N} \omega_{q,j} \lam^{-l/2}\big)\|^2_{L^\infty(K)} \leq C_{j,K,N} \lam^{-N-1+2j}
	$$
	for any $j,N \in \inte_+$, where $W_q \supset V_q^+ \cup V_q^-$ is an open neighborhood of $V_q^+ \cup V_q^-$.
\end{lemma}

Furthermore, the integral of the leading order term $\omega_{q,0}$ in the normal direction to the stable submanifold $V_q^+$ is computed in \cite{HelSj4}. 
\begin{lemma}\label{lem:eigenwkbcal}
	Fixing any point $x \in V_{q}^+$ and $\chi \equiv 1$ around $x$ compactly supported in $W$, we take any closed submanifold (possibly with boundary) $NV_{q,x}^+$ of $W$ intersecting transversally with $V_q^+$ at $x$. We have
	\begin{equation*}
	\lam^{\frac{\deg(q)}{2}}\int_{NV_{q,x}^+} e^{-\lam \g^+_{q}}\chi \omega_{q,0} = 1+\mathcal{O}(\lam^{-1}); \quad 
	\frac{\lam^{\frac{\deg(q)}{2}}}{\|e^{-\lam f_{ij}} \phi_{ij}(q) \|^2}\int_{NV_{q,x}^-} e^{-\lam \g^-_{q}}\chi \ast \omega_{q,0} = 1+\mathcal{O}(\lam^{-1}), 
	\end{equation*}
	for any point $x \in V_{q}^-$, with $NV_{q,x}^-$ intersecting transversally with $V_q^-$.
\end{lemma}

\subsubsection{WKB for homotopy operator}\label{sec:homotopy_wkb_result}
We recall the key estimate for the homotopy operator $H_{ij}$ proven in \cite[Section 4]{klchan-leung-ma}. Let $\gamma(t)$ be a flow line of $\nabla f_{ij}/ | \nabla f_{ij}|_{\dist_{ij}}$ starts at $\gamma(0)=x_{S}$ and $\gamma(T) = x_{E}$ for a fixed $T>0$ as shown in the following figure \ref{fig:situation}. 
\begin{figure}[h]
	\centering
	\includegraphics[scale = 0.25]{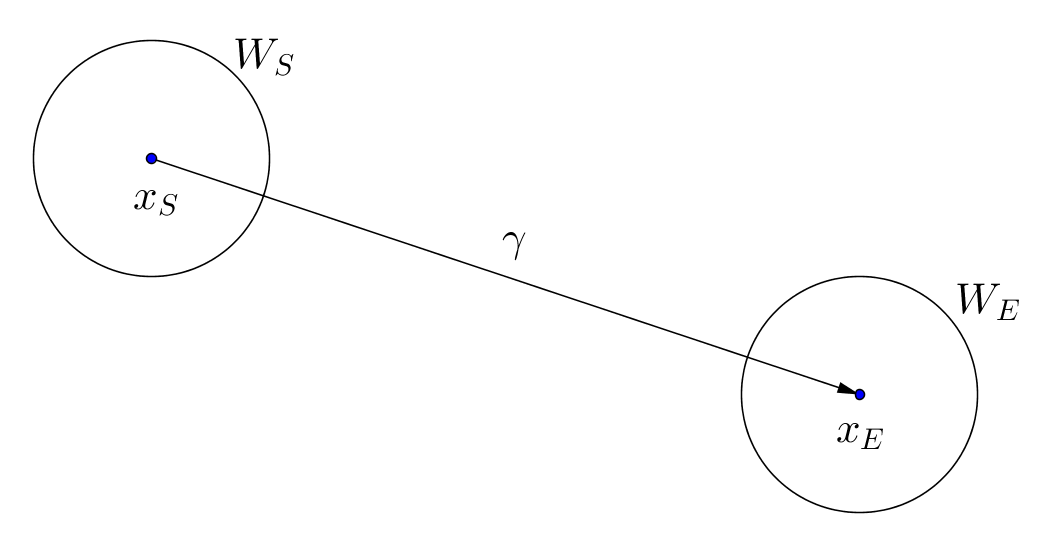}
	\caption{gradient flow line $\gamma$}\label{fig:situation}
\end{figure}
We consider an input form $\zeta_S$ defined in a neighborhood $W_S$ of $x_{S}$. Suppose we are given a WKB approximation of $\zeta_S$ in $W_S$, which is an approximation of $\zeta_S$ according to order of $\lam$ of the form
\begin{equation}\label{eqn:awkbexpansion}
\zeta_S \sim e^{-\lam \g_{\scalebox{.7}{$\scriptscriptstyle S$}}}(\omega_{S,0}+\omega_{S,1} \lam^{-1/2}+ \omega_{S,2} \lam^{-1} + \dots)
\end{equation}
which means we have $\lam_{j,0}>0$ such that when $\lam > \lam_{j,N,0}$ we have
$$
\|  e^{\lam\g^{}_{\scalebox{.7}{$\scriptscriptstyle S$}}}\nabla^j ( \zeta_S -  e^{-\lam \g^{}_{\scalebox{.7}{$\scriptscriptstyle S$}}} (\sum_{i=0}^{N} \omega_{S,i} \lam^{-i/2})) \|^2_{L^\infty(W_S)} \leq C_{j,N} \lam^{-N-1+2j},
$$
for any $j,N \in \inte_+$. We further assume that $g_S$ is a nonnegative Bott-Morse function in $W_S$ with zero set $V_S$ such that $\gamma$ is not tangent to $V_S$ at $x_S$. We consider the equation
\begin{equation}\label{eqn:homotopyeq}
\Delta_{ij} \zeta_E = (I-P_{ij})d_{ij}^*(\chi_S \zeta_S),
\end{equation}
where $\chi_S$ is a cutoff function compactly supported in $W_S$, $P_{ij}:\Omega^*(M)\rightarrow\Omega_{ij,<1}^*$ is the projection. We want to have a WKB approximation of $\zeta_E = H_{ij}(\chi_S \zeta_S)$
\begin{lemma}\label{lem:homotopywkb}
	For $\supp(\chi_S)$ small enough (the size only depends on $\g_S$ and $f_{ij}$), there is a WKB approximation of $\zeta_E$ in a small enough neighborhood $W_E$ of $x_E$, of the form
	$
	\zeta_E \sim e^{-\lam \g_{\scalebox{.7}{$\scriptscriptstyle E$}}}\lam^{-1/2}(\omega_{E,0}+\omega_{E,1}\lam^{-1/2}+\dots)
	$
	in the sense that we have $\lam_{j,0}>0$ such that when $\lam > \lam_{j,N,0}$ we have
	$$
	\|e^{\lam \g^{}_{\scalebox{.7}{$\scriptscriptstyle E$}}}\nabla^j\lbrace \zeta_E - e^{-\lam\g^{}_{\scalebox{.7}{$\scriptscriptstyle E$}}} (\sum_{i=0}^{N} \omega_{E,i}\lam^{-(i+1)/2} )\rbrace \|^2_{L^\infty(W_E)} \leq C_{j,N} \lam^{-N+2j}.
	$$
	Furthermore, the function $\g_E$ (only depending on $\g_S$ and $f_{ij}$) is a nonnegative function which is Bott-Morse in $W_E$ with zero set $V_E =( \bigcup_{-\infty<t<+\infty} \varsigma_t(V_S)) \cap W_E$ which is a closed submanifold in $W_E$, where $\varsigma_t$ is the $t$-time $\nabla f_{ij}/ |\nabla f_{ij}|^2$.
\end{lemma}

Finally, we have the following Lemma \ref{lem:wkb_integration_relation} from \cite{klchan-leung-ma} relating the integrals of $\omega_{S,0}$ and $\omega_{E,0}$. 

\begin{lemma}\label{lem:wkb_integration_relation}
	Using same notations in lemma \ref{lem:homotopywkb} and suppose $\chi_S$ and $\chi_E$ are cutoff functions supported in $W_S$ and $W_E$ respectively, then we have
	\begin{equation}
	\lam^{-\half}\int_{N_{x_E}} e^{-\lam \g^{}_{\scalebox{.7}{$\scriptscriptstyle E$}}} \chi_{E} \omega_{E,0} = (\int_{N_{x_S}} e^{-\lam \g^{}_{\scalebox{.7}{$\scriptscriptstyle S$}}}\chi_{S} \omega_{S,0})(1+\mathcal{O}(\lam^{-1})).
	\end{equation}
	Furthermore, suppose $\omega_{S,0}(x_S)\in \bigwedge^{top} N(V_S)^*_{x_S}$, we have $\omega_{E,0}(x_E) \in  \bigwedge^{top} N(V_E)^*_{x_E}$. Here $\bigwedge^{top} E$ refers to $\bigwedge^r E$ for a rank $r$ vector bundle $E$. Here $N_{v_S}$ and $N_{v_E}$ are any closed submanifold of $W_S$ and $W_E$ intersecting $V_S$ and $V_E$ transversally at $x_S$ and $x_E$ respectively.
	\end{lemma}

\subsection{Apriori Estimate}\label{sec:apriori_estimate}

\begin{notation}\label{not:abbrev_notation_ij}
	From now on, we will consider a fixed generic sequence $\vec{f} = (f_0,\dots,f_k)$ with corresponding sequence of critical points $\vec{q} = (q_{(k-1)k},\dots ,q_{01},q_{0k})$ and a fixed labeled ribbon $k$-tree $\tr$ such that $\dim(\mathcal{M}_\tr(\vec{f},\vec{q})) = 0$ (the dimension is given by formula \eqref{eqn:dimension_of_moduli_space}). We use $q_{ij}$ to denote a fixed critical point of $f_{ij}$. $\phi(q_{ij})$ associated to $q_{ij}$ is abbreviated by $\phi_{ij}$.
\end{notation}

\begin{notation}\label{not:tree_degree_valency_notation}
	For $\tr \in \tree{k}$ or $\ltree{k}$ with $\vec{q}$, we let $\simplex_{\tr}:= \prod_{v \in \lvertex} \simplex_{\val(v)}$ of dimension $\val(T):= \sum_{v \in \lvertex} \val(v)$, and we also let $\deg(\tr):= \sum_{i=0}^{k-1} \deg(q_{i(i+1)}) - |\tr^{[1]}| - \val(T)$. We inductively define a volume form $\nu_{\tr}$ on $\simplex_{\tr}$ for labeled ribbon tree $\tr \in \ltree{k}$ by: letting $\nu_{l} = dt_l \wedge \cdots \wedge dt_1$ on the $\simplex_l$; and for $v_r$ labeled with $1$ we split $\tr$ at $v_r$ into $\tr_2$ and $\tr_1$ such that $\tr_2, \tr_1 ,e_o$ is clockwisely oriented, then we take $\nu_{\tr} = \nu_{\tr_2} \wedge \nu_{\tr_1}$; and for $v_r$ labeled with $u$ we split $\tr$ at $v_r$ into $\tr_l,\dots,\tr_1$ clockwisely, and we take $\nu_\tr = \nu_{\tr_l} \wedge \cdots \wedge \nu_{\tr_1} \wedge \nu_{l}$. We should also write $\nu_{\tr}^{\vee}$ to be the polyvector field dual to $\nu_{\tr}$. 
\end{notation}

\begin{definition}\label{def:distance_function_of_a_tree}
	Given a labeled ribbon $k$-tree $\tr$ with $\vec{f}$ and $\vec{q}$ as above, we associate to it a length function $\hat{\tdist}_{\tr}$ on $ \mathfrak{M}(\tr):= \simplex_{\tr} \times M^{|\tr^{[0]}_{ni}|} \rightarrow \real_+$ \footnote{Here $\tr^{[0]}_{ni}$ is the set of all vertices besides incoming edges introduced in Definition \ref{def:labeled_k_tree}} with coordinates $(\vec{\mathbf{t}}_{\tr}, \hat{x}_{\tr})$ (where $\vec{\mathbf{t}}_{\tr} = (\mathbf{t}_v)_{v\in \lvertex}$ and $\hat{x}_{\tr} = (x_{v})_{v \in \tr^{[0]}_{ni}}$) inductively along the tree by the rules: 
	\begin{enumerate}
		\item for the unique tree with one edge $e$ numbered by $ij$, we take $\hat{\tdist}_{\tr}(x_{v_o}):= \tdist_{ij}(q_{ij},x_{v_o})$;
		\item when $v_r$ is labeled with $1$, we split $\tr$ at the root vertex $v_r$ into $\tr_2,\tr_1$. We notice that $\mathfrak{M}(\tr) = \mathfrak{M}(\tr_2) \times_M \mathfrak{M}(\tr_1) \times M_{v_o}$ (with coordinates $\vec{\mathbf{t}}_{\tr} = (\vec{\mathbf{t}}_{\tr_2} , \vec{\mathbf{t}}_{\tr_1})$, and $\hat{x}_{\tr} = (\hat{x}_{\tr_2},\hat{x}_{\tr_1}, x_{v_o}) $ such that $x_{\tr_2,v_r} = x_{\tr_1,v_r} = x_{v_r}$ in $M$) and we let $$\hat{\tdist}_{\tr}(\vec{\mathbf{t}}_{\tr},\hat{x}_{\digamma(\tr)}) = \hat{\tdist}_{ij}(x_{v_r},x_{v_o}) + \sum_{j=1}^2 \hat{\tdist}_{\tr_j}(\vec{\mathbf{t}}_{\tr_j},\hat{x}_{\tr_j})$$ if the numbering on $e_o$ is $ij$;
		\item when $v_r$ is labeled with $u$, we split $\tr$ at $v_r$ into $\tr_l ,\dots , \tr_1$ and we can write $\mathfrak{M}(\tr) = \mathfrak{M}_{\tr_l} \times_M \cdots \times_M \mathfrak{M}(\tr_1) \times_M (\simplex_{l} \times M_{v_r}) \times M_{v_o}$ where $l = \val(v_r)$. By writing coordinates $(\vec{\mathbf{t}}_{\tr_j}, \hat{x}_{\tr_j})$ for $\mathfrak{M}(\tr_j)$, $\mathbf{t}_{v_r} = (t_{v_r,l},\dots,t_{v_r,1})$ for $\simplex_l$, $x_{v_r}$ for $M_{v_r}$ and $x_{v_o}$ for $M_{v_o}$ satisfying $x_{\tr_l,v_r} = \actionmap_{t_{v_r,l}} (x_{v_r}) ,\cdots , x_{\tr_1,v_r} = \actionmap_{t_{v_r,1}} (x_{v_r})$, we let
		$$
		\hat{\tdist}_{\tr}(\vec{\mathbf{t}}_{\tr},\hat{x}_{\tr}) :=  \hat{\tdist}_{ij}(x_{v_r},x_{v_o}) + \sum_{j=1}^l 	\hat{\tdist}_{\tr_j} \big(\vec{\mathbf{t}}_{\tr_j} ,  \hat{x}_{\tr_j} \big)
		$$
		if the numbering on $e_o$ is $ij$. 
	\end{enumerate}
Fixing the outgoing point $x_{v_o} = q_{0k}$ giving coordinates $\vec{x}_{\tr} = (x_v)_{v \in \tr^{[0]}}$ for $M^{|\tr^{[0]}|}$, we let  $\tdist_\tr(\vec{\mathbf{t}}_{\tr},\vec{x}_\tr) := \hat{\tdist}_{\tr}(\vec{\mathbf{t}}_\tr, \vec{x}_\tr, q_{0k})$. 
	\end{definition}

\begin{example}\label{eg:example_on_distance_function}
	Suppose that $\tr$ is the labeled ribbon $2$-tree with two incoming vertices $v_{2}$ and $v_{1}$ joining to $v$ labeled with $u$ by $e_{12}$ and $e_{01}$, and $v$ is joining to the root vertex $v_r$ labeled with $u$ via $e$. Then we have $\simplex_{\tr} \times M^{|\tr^{[0]}_{ni}|} = \simplex_2 \times \mathbb{S}^1 \times M^3$ and $\hat{\tdist}_{\tr}(t_{v,2},t_{v,1},t_{v_r},x_{v},x_{v_r},x_{v_o}) =  \tdist_{02}(x_{v_r},x_{v_o}) + \tdist_{02}(x_{v},\sigma_{t_{v_r}}(x_{v_r})) + \tdist_{12}(q_{12},\sigma_{t_{v,2}}(x_v)) + \tdist_{01}(q_{01},\sigma_{t_{v,1}}(x_v))$. The following Figure \ref{fig:jumping_distance} shows the tree $\tr$ and its associated $\hat{\tdist}_\tr$.
	
	 \begin{figure}[h]
		\centering
		\includegraphics[scale = 0.5]{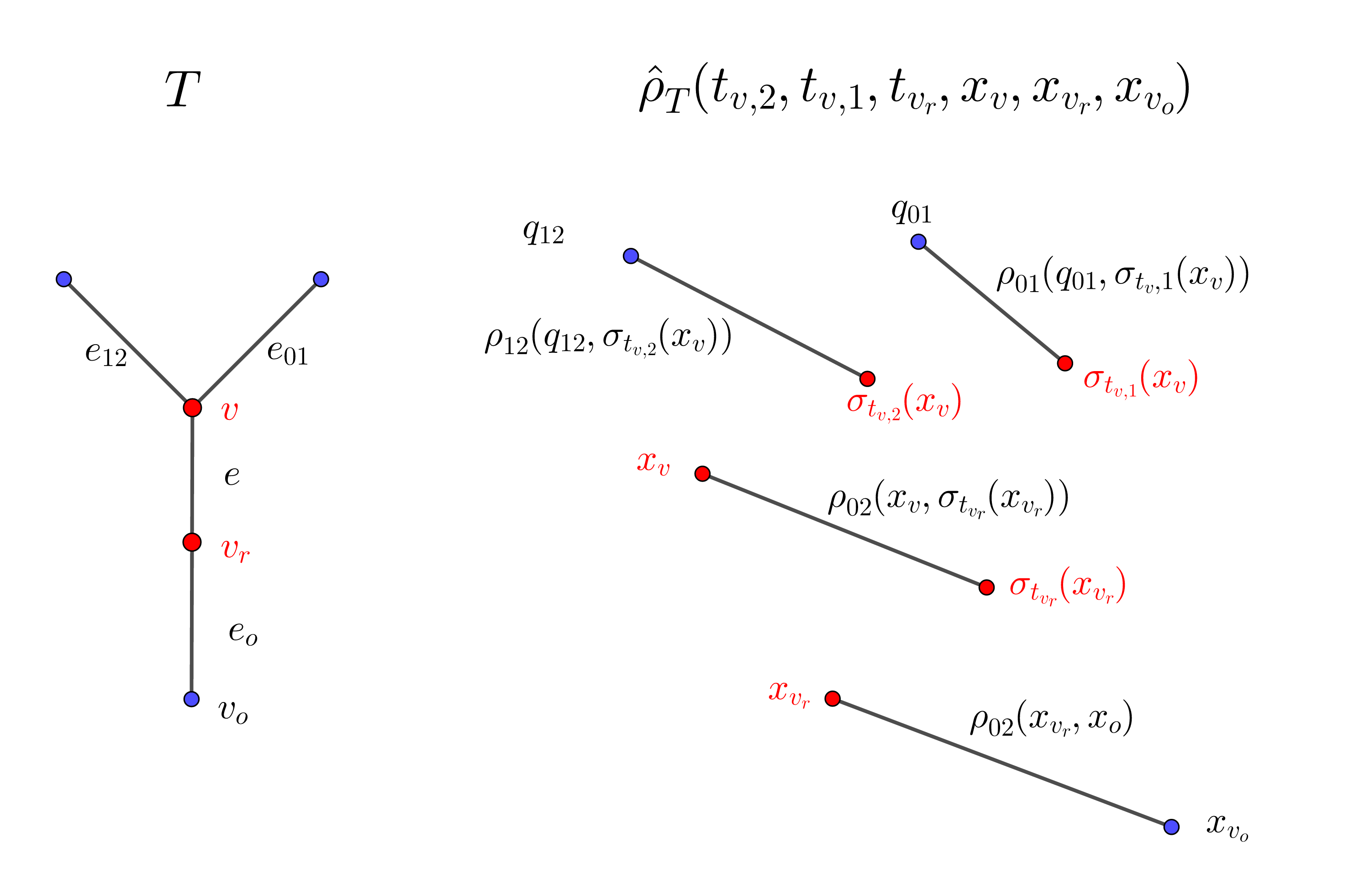}
		\caption{Distance function associated to $\tr$}\label{fig:jumping_distance}
	\end{figure}
	\end{example}

From its construction and Lemma \ref{lem:agmon_dist_flow_line}, we notice that $\tdist_{\tr}(\vec{\mathbf{t}}_{\tr},\vec{x}_{\tr}) \geq 0$ and equality holds if and only if for each edge $e$ numbered by $ij$ with $\partial_{in}(e) = v_1$ and $\partial_{o}(e) = v_2$, there is a generalized flow line of $\nabla f_{ij}$ joining $x_{v_1}$ to $\tilde{x}_{v_2}$, where $\tilde{x}_{v_2}= x_{v_2}$ when $v_2$ is labeled by $1$; and $\tilde{x}_{v_2} = \actionmap_{t_{v_2,j}}(x_{v_2})$ if $v_2$ is labeled by $u$ with and $e$ is the $j^{\text{th}}$ incoming edges of $v_2$ in the anti-clockwise orientation. Therefore, we have a generalized flow tree (with jumping) of type $(\tr,\vec{f},\vec{q})$ (which is a generalization of flow tree in Definition \ref{def:equivariant_morse_counting} by allow broken flow lines as in Definition \ref{lem:agmon_dist_flow_line}). With the condition that $\dim(\mathcal{M}(\vec{f},\vec{q})) = 0$ as mentioned in Notation \ref{not:abbrev_notation_ij}, we notice that every such generalized flow line is an actual flow line from the generic assumption \ref{def:moduli_space_as_intersection} for $\vec{f}$, because the expected dimension for flow tree with broken flow line is negative.

\begin{notation}\label{not:gradient_flow_tree_notation}
	We let $\ftree_1 , \dots, \ftree_{d}$ be the gradient flow tree of type $(\tr,\vec{f},\vec{q})$, such that each $\ftree_i$ is associated with a point $\mathbf{t}_{\ftree_i,v} \in \simplex_{\val(v)}$ (for $v \in \lvertex$) and $x_{\ftree_i,v} \in M$ (for $v \in \tr^{[0]}$) such that 
	\begin{enumerate}
		\item $x_{\ftree_i,v}$ is the starting point of a gradient flow line $\gamma_e$ associated to edge $e$ if $\partial_{in}(e) = v$, and we write $x_{\ftree_i,e,v} = x_{\ftree_i,v}$ in this case;
		\item $x_{\ftree_i,v}$ is the end point of the gradient flow line $\gamma_e$ if $v$ is labeled by $1$ if $\partial_o(e) =v$, and we write $x_{\ftree_i,e,v} = x_{\ftree_i,v}$ in this case;
		\item and $\actionmap_{t_{\ftree_i,v,j}} (x_{\ftree_i,v})$ is the end point of a gradient flow line $\gamma_e$ associated to $j^{th}$-edge $e$ clockwisely if $v$ is labeled by $u$ and $\partial_o(e) =v$, and we write $x_{\ftree_i,e,v} = \actionmap_{t_{\ftree_i,v,j}} (x_{\ftree_i,v})$ in this case.
		\end{enumerate}
	\end{notation} 

We consider a sequence of cut off functions $\vec{\chi}:= (\chi_{v})_{v \in \tr^{[0]}}$ such that $\chi_v $ compactly supported in a ball $U_v :=  B(x_v,r/2)$ of radius $r$ centered at a fixed point $x_v \in M$, and $(\vec{\varkappa}_{v})_{v \in \lvertex}$ with $\varkappa_v$ compactly support in a small neighborhood $\mathbf{C}_v$ containing a fixed $\mathbf{t}_v = (t_{v,\val(v)},\dots,t_{v,1}) \in \simplex_{\val(v)}$ such that the Riemannian distance between $\actionmap_{t_{j}}(x)$ and $\actionmap_{t'_{j}}(x)$ is strictly less than $r/2$ for any $j$ and any $x \in M$ and any $\mathbf{t}$ and $\mathbf{t}'$ in $\mathbf{C}_v$. 

\begin{definition}\label{def:cut_off_operations}
	With $\vec{\chi}$ and $\vec{\varkappa}$ as above, we define $\mathfrak{m}^{(e,v)}_{\vec{\chi},\vec{\varkappa}} \in \Omega^*(\simplex_{\tr_{e,v}}\times M)$ \footnote{recall that $\tr_{e,v}$ is introduced in Notation \ref{not:morse_function_associated_to_edges}} for each flag $(e,v) \in \digamma(\tr)$ inductively along $\tr$ by letting: 
	\begin{enumerate}
		\item for the incoming edge $e_{ij}$ with $\partial_o(e_{ij}) = v$, we take $\mathfrak{m}^{(e_{ij},v)}_{\vec{\chi},\vec{\varkappa}} = \phi_{ij}$;
		\item when we have $(e,v)$ with $\partial_{in}(e) = v$ with $v$ is labeled with $1$ with, we let $\tr_2,\tr_1$ to be subtrees with outgoing edges $e_2, e_1$ ending at $v$ such that $e_2,e_1,e$ clockwisely oriented. With coordinates $\vec{\mathbf{t}}_{\tr_{e,v}} = (\vec{\mathbf{t}}_{\tr_2}, \vec{\mathbf{t}}_{\tr_1})$ for $\simplex_{\tr}= \simplex_{\tr_2} \times \simplex_{\tr_1}$, we let $$\mathfrak{m}^{(e,v)}_{\vec{\chi},\vec{\varkappa}} (\vec{\mathbf{t}}_{\tr_{e,v}},x) =(-1)^{\varepsilon}\nu_{\tr_{e,v}} \chi_{v_r}(x) \big(\iota_{\nu_{\tr_2}^{\vee}}\mathfrak{m}^{(e_2,v)}_{\vec{\chi},\vec{\varkappa}}(\vec{\mathbf{t}}_{\tr_2} ,x) \big) \wedge \big( \iota_{\nu_{\tr_1}^{\vee}} \mathfrak{m}^{(e_1,v)}_{\vec{\chi},\vec{\varkappa}} (\vec{\mathbf{t}}_{\tr_1} ,x) \big),$$
		where $\varepsilon = \deg\big(\iota_{\nu_{\tr_2}^{\vee}}\mathfrak{m}^{(e_2,v)}_{\vec{\chi},\vec{\varkappa}}(\vec{\mathbf{t}}_{\tr_2} ,x) \big)+1$;
		\item when we have $v$ labeled with $u$, we let $\tr_l,\dots,\tr_1$ be subtrees with outgoing edges $e_l,\dots,e_1$ ending at $v$ with $e_l,\dots,e_1,e$ clockwisely oriented. We let $$\mathfrak{m}^{(e,v)}_{\vec{\chi},\vec{\varkappa}}(\vec{\mathbf{t}}_{\tr_{e,v}},x) = \nu_{\tr_{e,v}} \chi_{v}(x) \varkappa_{v}(\mathbf{t}_v) \actionmap_{t_{v,l}}^*\big(\iota_{w_{v,l} \wedge \nu_{\tr_l}^{\vee}}\mathfrak{m}^{(e_l,v)}_{\vec{\chi},\vec{\varkappa}}(\vec{\mathbf{t}}_{\tr_l} ,x)\big)\wedge\cdots \wedge \actionmap_{t_{v,1}}^*\big(\iota_{w_{v,1} \wedge \nu_{\tr_1}^{\vee}}\mathfrak{m}^{(e_1,v)}_{\vec{\chi},\vec{\varkappa}}(\vec{\mathbf{t}}_{\tr_1} ,x)\big),$$
		where $t_{v,l},\dots,t_{v,1}$ is the coordinates for $\simplex_{\val(v)}$ and $\vec{\mathbf{t}}_{\tr_{e,v}} = (\vec{\mathbf{t}}_{\tr_l},\dots,\vec{\mathbf{t}}_{\tr_1},t_{v,l},\dots,t_{v,1})$, and $w_{v,j} = \actionmap_*(\dd{t_{v,j}})$;
		\item for an edge $e$ numbered by $ij$ with $\partial_{in}(e) = v_0$ and $\partial_{o}(e) = v_1$ with $v_1$ not being the outgoing vertex $v_o$, we let $\mathfrak{m}^{(e,v_1)}_{\vec{\chi},\vec{\varkappa}} =  d^*_{ij} G_{ij}(\mathfrak{m}^{(e,v_0)}_{\vec{\chi},\vec{\varkappa}})$ where $G_{ij}$ is introduced in Definition \ref{def:operation_associated_to_labeled_tree};
		\item for the outgoing edge $e_o$ with $\partial_{in}(e_o) = v_r$ and $\partial_o(e_o) = v_o$, we take  $\mathfrak{m}^{\tr}_{\vec{\chi},\vec{\varkappa}} = \mathfrak{m}^{(e_o,v_o)}_{\vec{\chi},\vec{\varkappa}}  = \mathfrak{m}^{(e_o,v_r)}_{\vec{\chi},\vec{\varkappa}}$.
		\end{enumerate}
	\end{definition}

\begin{example}\label{eg:example_for_cut_off_operation}
	We the tree $\tr$ described in the previous Example \ref{eg:example_on_distance_function}, we have $\mathfrak{m}^{(e,v)}_{\vec{\chi},\vec{\varkappa}} (t_{v,2},t_{v,1},x_v) = \chi_{v}(x_v) \varkappa_v(t_{v,2},t_{v,1}) dt_{v,2} dt_{v,1} \actionmap_{t_{v_2}}^*(\iota_{w_{v,2}}\phi_{02})(x_v) \wedge \actionmap_{t_{v_1}}^*(\iota_{w_{v,1}}\phi_{01})(x_v)$, $\mathfrak{m}^{(e,v_r)}_{\vec{\chi},\vec{\varkappa}} = d^*_{02} G_{02} (\mathfrak{m}^{(e,v)}_{\vec{\chi},\vec{\varkappa}})$ ($d^*_{02}G_{02}$ only acting on the component $M$) and  $$\mathfrak{m}^{(e_o,v_r)}_{\vec{\chi},\vec{\varkappa}}(t_{v,2},t_{v,1}, t_{v_r},x_{v_r})  = \chi_{v_r}(x_{v_r}) \varkappa(t_{v_r})  dt_{v,2} dt_{v,1} dt_{v_{r}} \actionmap_{t_{v_r}}^*(\iota_{w_{v_r} \wedge \dd{t_{v,1}} \wedge \dd{t_{v,2}} } \mathfrak{m}^{(e,v_r)}_{\vec{\chi},\vec{\varkappa}}) (x_{v_r}),$$ and finally we have $\mathfrak{m}^{\tr}_{\vec{\chi},\vec{\varkappa}} =    \mathfrak{m}^{(e_o,v_r)}_{\vec{\chi},\vec{\varkappa}}$. 
	\end{example}

We take a collection $\{ \vec{\chi}_{\mathbf{i}} \}_{\mathbf{i} \in \mathscr{I}}$ and $ \{ \vec{\varkappa}_{\mathbf{j}} \}_{\mathbf{j} \in \mathscr{J}}$ such that $\vec{\chi}_{\mathbf{i}} = (\chi_{i,v})_{\substack{i \in \mathscr{I}_v \\ v \in \tr^{[0]} }}$ and $\vec{\varkappa}_{\mathbf{j}} = (\varkappa_{j,v})_{\substack{j \in \mathscr{J}_v \\ v \in \lvertex}}$ and such that every collection $\{\chi_{i,v} \}_{i \in \mathscr{I}_v}$ and $\{ \varkappa_{j,v} \}_{j \in \mathscr{J}_v}$ is a partition of unity for $M_{v}$ and $\simplex_{\val(v)}$ respectively (Here we use the notation $\mathscr{I} = \prod_{v\in \tr^{[0]}} \mathscr{I}_v$ and $\mathscr{J} = \prod_{v \in \tr^{[0]}} \mathscr{J}_{v}$). With the cut off construction in Definition \ref{def:cut_off_operations} and the Definition \ref{def:operation_associated_to_labeled_tree}, we have
\begin{equation}\label{eqn:cut_off_operation_summation}
\int_{M} m_{k,\tr}(\lam) (\phi_{(k-1)k},\dots,\phi_{01}) \wedge \frac{\ast e^{-2\lam f_{0k}}\phi_{0k}}{\| e^{-\lam f_{0k}} \phi_{0k}\|^2} = \sum_{\mathbf{i} \in \mathscr{I}} \sum_{\mathbf{j} \in \mathscr{J}} \int_{\simplex_{\tr} \times M} \mathfrak{m}^{\tr}_{\vec{\chi}_{\mathbf{i}},\vec{\varkappa}_{\mathbf{j}} } \wedge \frac{\ast e^{-2\lam f_{0k}} \phi_{0k}}{\| e^{-\lam f_{0k}} \phi_{0k}\|^2}.
\end{equation}

\begin{lemma}\label{lem:apriori_lemma}
	We fix a point $(\vec{\mathbf{t}}_{\tr}, \vec{x}_{\tr})$ in $\mathfrak{M}(\tr)$ with the cut off functions $\vec{\chi}$ and $\vec{\varkappa}$ and $\mathfrak{m}^{\tr}_{\vec{\chi},\vec{\varkappa}}$ as before Definition \ref{def:cut_off_operations}, for any $\epsilon > 0$ we have $\lam_0(\epsilon)$ and small enough radius $r = r(\epsilon)$ of cut off functions (which is described before Definition \ref{def:cut_off_operations}) such that when $\lam > \lam_0$ we have the norm estimate 
	$$\| \mathfrak{m}^{\tr}_{\vec{\chi},\vec{\varkappa}}\wedge \frac{\ast e^{-2\lam f_{0k}} \phi_{0k}}{\| e^{-\lam f_{0k}} \phi_{0k}\|^2} \|_{C^j(\simplex_{\tr} \times M)} \leq C_{j,\epsilon}  e^{-\lam ( \tdist_{\tr}(\vec{\mathbf{t}}_{\tr},\vec{x}_{\tr}) - b_{\tr}\epsilon ) },$$
	for any $j \in \inte_+$ (Here we fix an arbitrary metric on the simplices $\simplex_l$'s), where $b_\tr$ is a constant depending the combinatorics of $\tr$.
	\end{lemma}

\begin{proof}
	We prove by induction along the tree $\tr$ that for each flag $(e,v)$ with $\partial_o(e) = v \neq v_o$ we have
	$$
	\|  \mathfrak{m}^{(e,v)}_{\vec{\chi},\vec{\varkappa}}\|_{C^j(\simplex_{\tr_{e,v}} \times U_v)} \leq C_{j,\epsilon, \vec{\chi},\vec{\varkappa}} \exp \left( -\lam ( \hat{\tdist}_{\tr_{e,v}}(\vec{\mathbf{t}}_{\tr_{e,v}},\hat{x}_{\tr_{e,v}}) - b_{\tr_{e,v}}\epsilon ) \right),
	$$
	where $U_v = B(x_v,r/2)$, for any points $\vec{\mathbf{t}}_{\tr} \in \simplex_\tr$, $\hat{x}_{\tr} \in M^{|\tr^{[0]}_{ni}|}$ with the assoicated cut off functions $\vec{\varkappa}$ and $\vec{\chi}$ with small enough $r$. The initial case follows from the estimate in Lemma \ref{lem:eigenestimate2}. For induction we consider an edge $e$ with $\partial_{in}(e) = v$ and $\partial_{o}(e) = \tilde{v}$.  We take subtrees (of $\tr$) $\tr_l,\dots,\tr_1$ with edges $e_l,\dots,e_1$ attached to $v$ such that $e_l,\dots,e_1,e$ is clockwisely oriented. There are two cases. 
	
	The first case is when $v$ is labeled with $1$ and we have $l=2$. In this case we have the estimate 
	$$
	\|\mathfrak{m}^{(e_2,v)}_{\vec{\chi},\vec{\varkappa}} \wedge \mathfrak{m}^{(e_1,v)}_{\vec{\chi},\vec{\varkappa}} \|_{C^j(\simplex_{\tr_{e_2,v}} \times \simplex_{\tr_{e_1,v}} \times U_v)} \leq C_{j,\epsilon, \vec{\chi},\vec{\varkappa}} \exp\left(-\lam \big(\hat{\tdist}_{\tr_2}(\vec{\mathbf{t}}_{\tr_{2}} , \hat{x}_{\tr_2}) + \hat{\tdist}_{\tr_1} (\vec{\mathbf{t}}_{\tr_{1}} , \hat{x}_{\tr_1}) - b_{\tr_{e,v}}\epsilon \big) \right)
	$$ 
	by choosing $b_{\tr_{e,v}} \geq b_{\tr_1} + b_{\tr_2}$, where we require $x_{\tr_1,v} = x_{\tr_2,v} = x_{v}$ in the R.H.S. of the above equation. Assuming that $e$ is numbered by $ij$, and we apply the Lemma \ref{lem:resolventlemma} to the term $\mathfrak{m}^{(e,\tilde{v})}_{\vec{\chi},\vec{\varkappa}}=  d^*_{ij} G_{ij} \big(\chi_{v} \mathfrak{m}^{(e_2,v)}_{\vec{\chi},\vec{\varkappa}} \wedge \mathfrak{m}^{(e_1,v)}_{\vec{\chi},\vec{\varkappa}} \big)$ (we choose smaller $r$ if necessary) we obtain the estimate
	$$
	\|  d^*_{ij} G_{ij} \big(\chi_{v} \mathfrak{m}^{(e_2,v)}_{\vec{\chi},\vec{\varkappa}} \wedge \mathfrak{m}^{(e_1,v)}_{\vec{\chi},\vec{\varkappa}} \big) \|_{C^j (\simplex_{\tr_{e,\tilde{v}}} \times U_{\tilde{v}})} \leq C_{j,\epsilon, \vec{\chi},\vec{\varkappa}} \exp \left( - \lam \big( \hat{\tdist}_{\tr_{e,\tilde{v}}} (\vec{\mathbf{t}}_{\tr_{e,\tilde{v}}},\hat{x}_{\tr_{e,\tilde{v}}}) - b_{\tr_{e,\tilde{v}}}\epsilon \big) 
	\right),
	$$
	by taking $b_{\tr_{e,\tilde{v}}} \geq b_{\tr_{e,v}} + 1$ which is the desired estimate.
	
	The second case is when $v$ is labeled with $u$, and we have the estimate
	\begin{multline*}
	\|   \actionmap^*_{t_{l}} \big( \iota_{w_{v,l}\wedge  \nu_{\tr_l}^\vee} \mathfrak{m}^{(e_l,v)}_{\vec{\chi},\vec{\varkappa}}  \big) \wedge \cdots \wedge \actionmap^*_{t_1} \big( \iota_{w_{v,1}\wedge \nu_{\tr_1}^\vee} \mathfrak{m}^{(e_1,v)}_{\vec{\chi},\vec{\varkappa}}  \big) \|_{C^j(\prod_{j=1}^l \simplex_{\tr_{j}} \times \mathbf{C}_v\times U_v)} \\ \leq C_{j,\epsilon, \vec{\chi},\vec{\varkappa}} \exp \left( -\lam \big( \sum_{j=1}^l \hat{\tdist}_{\tr_{j}} (\vec{\mathbf{t}}_{\tr_j}, \hat{x}_{\tr_{j}})  - b_{\tr_{e,v}} \epsilon \big) \right),
	\end{multline*}
	using the induction hypothesis and by taking $b_{\tr_{e,v}} \geq l+  \sum_{j=1}^l b_{\tr_j} $, for $(t_l,\dots,t_1)$ varying in  small enough neighborhood $\mathbf{C}_v$ of $(t_{v,l},\dots,t_{v,1})$ ($\mathbf{C}_v$ introduced in the paragraph before Definition \ref{def:cut_off_operations}), where we require that the identity $x_{\tr_j,v} =  \actionmap_{t_{v,j}}(x_{v})$ on the R.H.S. as in the Definition \ref{def:distance_function_of_a_tree}. By applying $d^{*}_{ij} G_{ij}$ (if $e$ is numbered by $ij$) to the term $\mathfrak{m}^{(e,v)}_{\vec{\chi},\vec{\varkappa}} = \nu_{\tr_{e,v}} \chi_v \varkappa_v \actionmap^*_{t_{l}} \big( \iota_{w_{v,l} \wedge \nu_{\tr_l}^\vee} \mathfrak{m}^{(e_l,v)}_{\vec{\chi},\vec{\varkappa}}  \big) \wedge \cdots \wedge \actionmap^*_{t_1} \big( \iota_{w_{v,1} \wedge \nu_{\tr_1}^\vee} \mathfrak{m}^{(e_1,v)}_{\vec{\chi},\vec{\varkappa}} \big)$ as in Definition \ref{def:cut_off_operations}, and using Lemma \ref{lem:resolventlemma} again we have the desired estimate
	$$
	\| \mathfrak{m}^{(e,\tilde{v})}_{\vec{\chi},\vec{\varkappa}} \|_{C^j(\simplex_{\tr_{e,\tilde{v}}} \times U_{\tilde{v}})} \leq C_{j,\epsilon, \vec{\chi},\vec{\varkappa}} \exp \left( -\lam \big( \hat{\tdist}_{\tr_{e,\tilde{v}}}(\vec{\mathbf{t}}_{\tr_{e,\tilde{v}}}, \hat{x}_{\tr_{e,\tilde{v}}}) - b_{\tr_{e,\tilde{v}}} \epsilon \big)\right),
	$$
	where we take $b_{\tr_{e,\tilde{v}}} \geq b_{\tr_{e,v}}+1$. 
	
	To obtain the statement of the Lemma, we observe that if $\tr_l ,\cdots,\tr_1$ are the incoming trees joining to the root vertex we have
	$$
	\| \mathfrak{m}^{(e_o,v_o)}_{\vec{\chi},\vec{\varkappa}} \|_{C^j(\simplex_{\tr} \times U_{v_r})} \leq C_{j,\epsilon,\vec{\chi},\vec{\varkappa}} \exp \left( -\lam \big(\sum_{j=1}^l \hat{\tdist}_{\tr_j}(\vec{\mathbf{t}}_{\tr_j},\hat{x}_{\tr_j}) - b_{\tr_{e_o,v_o}} \epsilon \big)  \right)
	$$
	in a small enough neighborhood $U_{v_r}$ of $x_{v_r}$, where we have $l=2$ and $x_{\tr_2,v_r} = x_{\tr_1,v_r} = x_{v_r}$ in R.H.S. as in the first case with $v_r$ labeled with $1$, and $x_{\tr_j,v_r} = \actionmap_{t_{v_r,j}}(x_{v_r})$ in R.H.S. as in the second case that $v_r$ is labeled with $u$. The Lemma follows from the estimate for $\mathfrak{m}^{(e_o,v_o)}_{\vec{\chi},\vec{\varkappa}}$ and that for $\frac{\ast e^{-2\lam f_{0k}}\phi_{0k}}{\| e^{-\lam f_{0k}} \phi_{0k}\|^2}$ in Remark \ref{rem:eigenwkbremark}. 
	\end{proof}

The above Lemma allows us to estimate the terms $\mathfrak{m}^{\tr}_{\vec{\chi},\vec{\varkappa}}$ appearing in the R.H.S., and from the discussion after Example \ref{eg:example_on_distance_function} we notice that it is closely related to gradient flow tree of type $\tr$. With the gradient flow trees $\ftree_i$'s as in Notation \ref{not:gradient_flow_tree_notation}, we assume there are open neighborhoods $D_{\ftree_i,v}$ and $W_{\ftree_i,v}$ of $x_{\ftree_i,v}$ for $v \in \tr^{[0]} $ such that $\overline{D_{\ftree_i,v}} \subset W_{\ftree_i,v}$ together with $\chi_{\ftree_i,v} \equiv 1$ on $\overline{D_{\ftree_i,v}}$ which is compactly supported in $W_{\ftree_i,v}$ giving $\vec{\chi}_{\ftree_i} = (\chi_{\ftree_i,v})_{v \in \tr^{[0]}}$. Similarly, we also assume there are open neighborhoods $\mathbf{C}_{\ftree_i,v}$ and $\mathbf{E}_{\ftree_i,v}$ of $\mathbf{t}_{\ftree_i,v}$ in $\simplex_{\val(v)}$ satisfying $\overline{\mathbf{C}_{\ftree_i,v}} \subset \mathbf{E}_{\ftree_i,v}$ together with $\varkappa_{\ftree_i,v} \equiv 1$ on $\overline{\mathbf{C}_{\ftree_i,v}}$ which is compactly supported in $\mathbf{E}_{\ftree_i,v}$ giving $\vec{\varkappa}_{\ftree_i} = (\varkappa_{\ftree_i,v})_{v \in \lvertex}$. We should further prescribe the size of these neighborhood $W_{\ftree_i,v}$'s and $\mathbf{E}_{\ftree_i,v}$ in the upcoming Section \ref{sec:wkb_approximation_method} which is defined along the gradient tree $\ftree_i$'s together with the WKB approximation \footnote{Roughly speaking, these are the open subsets that WKB approximation for $\mathfrak{m}^{(e,v)}_{\vec{\chi},\vec{\varkappa}}$ can be constructed. These open subsets does not depend on $\mathfrak{m}^{(e,v)}_{\vec{\chi},\vec{\varkappa}}$ but rather depend on the geometry of gradient flow tree $\ftree_i$'s when applying Lemma \ref{lem:eigenwkb} and Lemmma \ref{lem:homotopywkb} along $\ftree_i$'s.}. By writing $\overline{\vec{D}_{\ftree_i}} =  \prod_{v \in \tr^{[0]}} \overline{D_{\ftree_i,v}}$ and $\overline{\vec{\mathbf{C}}_{\ftree_i}} = \prod_{v \in \lvertex} \overline{\mathbf{C}_{\ftree_i,v}}$, we have $\tdist_\tr \geq c>0$ for some constant $c$ outside $\bigcup_{i=1}^d \overline{\vec{\mathbf{C}}_{\ftree_i}} \times \overline{\vec{D}_{\ftree_i}}$ by continuity of $\tdist_{\tr}$ and the discussion after Example \ref{eg:example_on_distance_function}. As a result, we can fix a small enough $\epsilon$ (and the associated $r(\epsilon)$) such that $b_{\tr} \epsilon < c/2$. The following Figure \ref{fig:tree_open_neighborhood} show the situation for these open subsets $W_{\Gamma_i,v}$'s and $\mathbf{E}_{\Gamma_i,v}$'s for the tree in Example \ref{eg:example_on_distance_function}.

\begin{figure}[h]
\centering
\includegraphics[scale = 0.6]{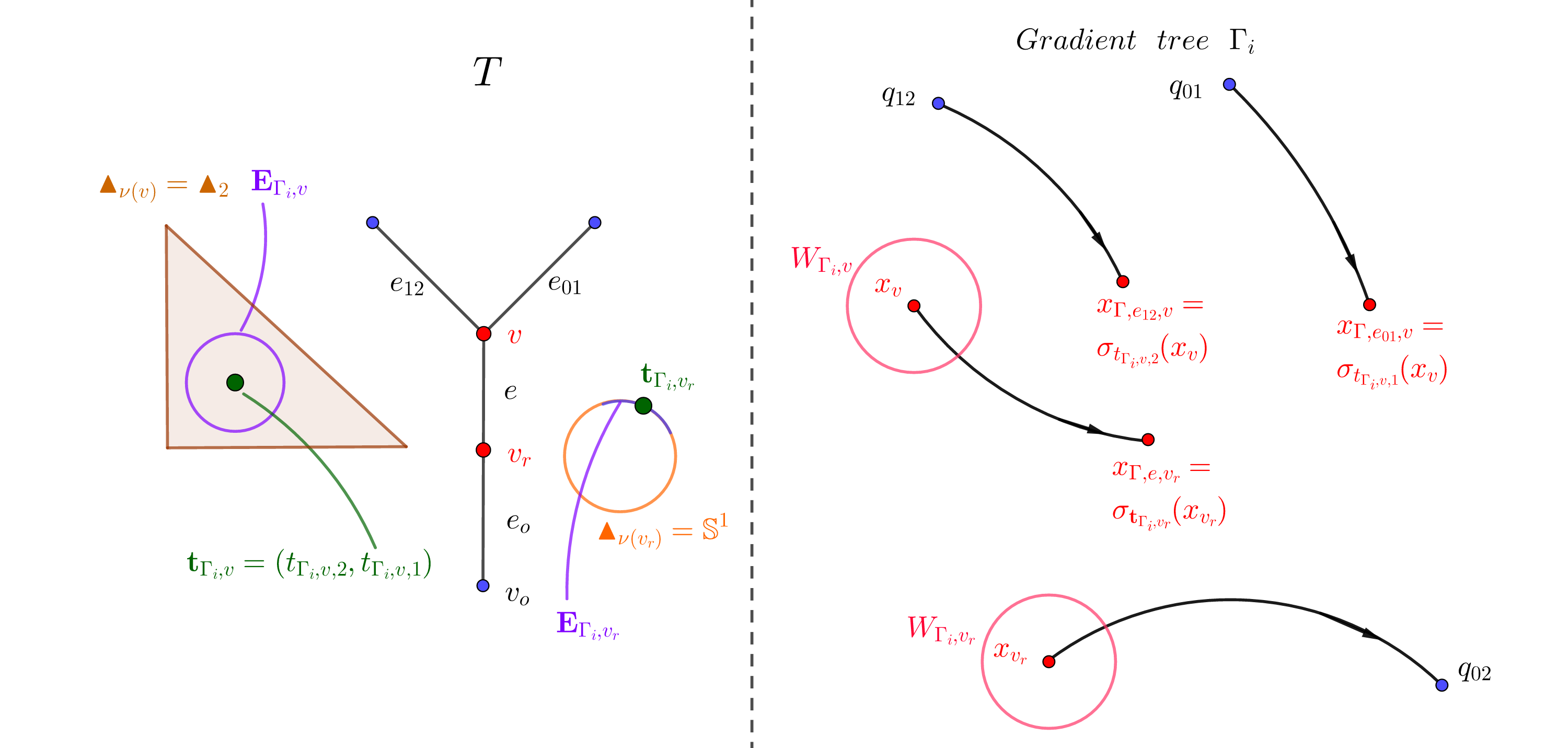}
\caption{Open subsets near gradient tree $\Gamma_i$}\label{fig:tree_open_neighborhood}
\end{figure}

We can take a finite collection $\{\vec{\chi}_{\mathbf{i}}\}_{\mathbf{i} \in \mathscr{I}}$ and $\{\vec{\varkappa}_{\mathbf{j}} \}_{\mathbf{j} \in \mathscr{J}}$ in the paragraph before Lemma \ref{lem:apriori_lemma} such that $\{\vec{\chi}_{\mathbf{i}}\}_{\mathbf{i} \in \mathscr{I}} \cup \{\vec{\chi}_{\ftree_1},\dots,\vec{\chi}_{\ftree_d} \}$ forms a partition of unity of $M^{|\tr^{[0]}|}$ and finite collection $\{\vec{\varkappa}_{\mathbf{j}} \}_{\mathbf{j} \in \mathscr{J}} \cup \{\vec{\varkappa}_{\ftree_1},\dots, \vec{\varkappa}_{\ftree_d} \}$ forms a partition of unity of $\simplex_{\tr}$ respectively, further satisfying $\big( \text{Supp}(\vec{\chi}_{\mathbf{i}}) \times \text{Supp}(\vec{\varkappa}_{\mathbf{j}}) \big) \cap \overline{\vec{\mathbf{C}}_{\ftree_i}} \times \overline{\vec{D}_{\ftree_i}} = \emptyset$ for each flow tree $\ftree_i$ and any $\mathbf{i},\mathbf{j}$. Therefore we have the estimate $\| \mathfrak{m}^{\tr}_{\vec{\chi}_{\mathbf{i}},\vec{\varkappa}_{\mathbf{j}}} \wedge \frac{\ast e^{-2\lam f_{0k}}\phi_{0k}}{\| e^{-\lam f_{0k}} \phi_{0k}\|^2} \|_{C^0(\simplex_{\tr} \times M)} \leq C_{\epsilon,\vec{\chi}_{\mathbf{i}},\vec{\varkappa}_{\mathbf{j}}} e^{-\lam c/2}$. As a conclusion of this Section \ref{sec:apriori_estimate}, we have 
\begin{equation}\label{eqn:apriori_cut_off_near_flow_tree}
\int_{M} m_{k,\tr}(\lam) (\phi_{(k-1)k},\dots,\phi_{01}) \wedge \frac{\ast e^{-2\lam f_{0k}}\phi_{0k}}{\| e^{-\lam f_{0k}} \phi_{0k}\|^2} = \sum_{i=1}^d \int_{\simplex_{\tr} \times M} \mathfrak{m}^{\tr}_{\vec{\chi}_{\ftree_i},\vec{\varkappa}_{\ftree_i}} \wedge \frac{\ast e^{-2\lam f_{0k}}\phi_{0k}}{\| e^{-\lam f_{0k}} \phi_{0k}\|^2} + \mathcal{O}(e^{-\lam c/2}),
\end{equation}
where $\mathcal{O}(e^{-\lam c/2})$ refers to function in $\lam$ bounded by $C e^{-\lam c/2}$ for some $C$. This cut off the contribution to integral near the gradient flow trees $\ftree_i$'s. 

\subsection{WKB approximation method}\label{sec:wkb_approximation_method}
\subsubsection{WKB expansion for $\mathfrak{m}^{(e,v)}_{\vec{\chi},\vec{\varkappa}}$} 
We fix a particular gradient flow tree $\ftree = \ftree_i$ (we omit $i$ in our notations for the rest of this paper) and compute the contribution from the integral $\int_{\simplex_{\tr} \times M} \mathfrak{m}^{\tr}_{\vec{\chi}_{\ftree},\vec{\varkappa}_{\ftree}} \wedge \frac{e^{-2\lam f_{ij}}\ast \phi_{0k}}{\| e^{-\lam f_{0k}} \phi_{0k}\|^2}$ in the above equation \ref{eqn:apriori_cut_off_near_flow_tree} using techniques from \cite[Section 3]{klchan-leung-ma}. 

We inductively define the open subset $W_{e,v} \subset M$ and $\mathbf{E}_{v}$ of $\mathbf{t}_v$ along the tree $\ftree$, together with a WKB expansion of $\mathfrak{m}^{(e,v)}_{\vec{\chi},\vec{\varkappa}}$ in $\vec{\mathbf{E}}_{\tr_{e,v}} \times W_{e,v} = \prod_{v \in \lvertex_{e,v}} \mathbf{E}_{v} \times W_{e,v}$ \footnote{Here $\tr_{e,v}$ is the combinatorial subtree of $\tr$ as in Notation \ref{not:morse_function_associated_to_edges}.} for each flag $(e,v)$ of $\tr$ 
\begin{equation}
\mathfrak{m}^{(e,v)}_{\vec{\chi},\vec{\varkappa}} \sim \lam^{r_{e,v}} e^{-\lam \g_{e,v}} \big(\omega_{(e,v),0} + \omega_{(e,v),1}\lam^{-\half} + \cdots \big),
\end{equation}
which is a norm estimate (here we fix arbitrary metric on $\simplex_{l}$ as before) in the sense of Lemma \ref{lem:homotopywkb}, where $g_{e,v} \in \mathcal{C}^{\infty}(\vec{\mathbf{E}}_{\tr_{e,v}} \times W_{e,v})$ is non-negative Bott-Morse function with zero set $V_{e,v} \subset \vec{\mathbf{E}}_{\tr_{e,v}} \times W_{e,v}$ and $\omega_{(e,v),i} \in \Omega^*(\vec{\mathbf{E}}_{\tr_{e,v}} \times W_{e,v})$ as follows: 
\begin{enumerate}
	\item for the incoming edges $e_{ij}$ with $\partial_{o}(e_{ij}) = v$, we define $W_{e_{ij},v}$ to be a open subset of $x_{\ftree,e_{ij},v}$ (We use the notation as in Notation \ref{not:gradient_flow_tree_notation}) together with the WKB expansion for $\phi_{ij}$ in $W_{e_{ij},v}$ from Lemma \ref{lem:eigenwkb}, with $r_{e_{ij},v} = \frac{\deg(q_{ij})}{2}$ and $\g_{e_{ij},v} = \g^+_{q_{ij}}$. In this case we have $V_{e_{ij},v} = V_{q_{ij}}^+ \cap W_{e_{ij},v}$ being the stable submanifold;
	
	\item for $(e,v)$ with $\partial_{in}(e) = v$ with $v$ is labeled with $1$, we let $\tr_2,\tr_1$ to be subtrees with outgoing edges $e_2, e_1$ ending at $v$ such that $e_2,e_1,e$ clockwisely oriented, we let $\vec{\mathbf{E}}_{\tr_{e,v}} =  \vec{\mathbf{E}}_{\tr_2} \times \vec{\mathbf{E}}_{\tr_1}$ and $W_{e,v} = W_{e_2,v} \cap W_{e_1, v}$, with the product WKB expansion as
	$$
	(-1)^{\varepsilon} \chi_{v} \mathfrak{m}^{(e_2,v)}_{\vec{\chi},\vec{\varkappa}} \wedge \mathfrak{m}^{(e_1,v)}_{\vec{\chi},\vec{\varkappa}} \sim \lam^{r_{e,v}} e^{-\lam \g_{e,v}} \big(\omega_{(e,v),0} + \omega_{(e,v),1}\lam^{-\half} + \cdots \big)
	$$
	by taking $\lam^{r_{e,v}} = \lam^{r_{e_2,v} + r_{e_1,v}}$, $\g_{e,v} = \g_{e_2,v} + \g_{e_1,v}$ and $\omega_{(e,v),l} = \sum_{i+j = l} \chi_{v} \omega_{(e_2,v),i} \wedge \omega_{(e_1,v),j}$ (Here $\varepsilon$ is given (2) in Definition \ref{def:cut_off_operations}).  In this case we have $\g_{e,v}$ being a non-negative Bott-Morse function in $\vec{\mathbf{E}}_{\tr_{e,v}}\times W_{e,v}$ with zero set $V_{e,v} = (V_{e_2,v} \times \vec{\mathbf{E}}_{\tr_1})\cap (V_{e_1,v} \times \vec{\mathbf{E}}_{\tr_1})$;
	
	\item when we have $v$ labeled with $u$, we let $\tr_l,\dots,\tr_1$ be subtrees with outgoing edges $e_l,\dots,e_1$ ending at $v$ with $e_l,\dots,e_1,e$ clockwisely oriented, we let $\vec{\mathbf{E}}_{\tr_{e,v}} = \prod_{j=1}^l \vec{\mathbf{E}}_{\tr_j} \times \mathbf{C}_v$ and take $W_{e,v}$ (Here $\mathbf{C}_v$ is neighborhood of $\mathbf{t}_{\ftree,v}$, and $W_{e,v}$ is a neighborhood of $x_{\ftree,v} = x_{\ftree,e,v}$) such that $\actionmap_{t_{j}} (W_{e,v}) \subset W_{e_j,v}$ for each $j=1,\dots,l$ for $(t_l,\dots,t_1) \in \mathbf{C}_v$. Therefore we have the WKB expansion $\mathfrak{m}^{(e,v)}_{\vec{\chi},\vec{\varkappa}} \sim \lam^{r_{e,v}} e^{-\lam \g_{e,v}} \big(\omega_{(e,v),0} + \omega_{(e,v),1}\lam^{-\half} + \cdots \big)$ by taking $r_{e,v} = \sum_{j=1}^l r_{e_j,v}$, $g_{e,v} = \sum_{j=1}^l \tau_j^* (g_{e_j,v})$ and 
	$$ \omega_{(e,v),m} = \sum_{i_l + \cdots + i_1 = m} \nu_{\tr_{e,v}}\chi_{v} \varkappa_{v} \big( \iota_{ \dd{t_{v,l}} \wedge \nu_{\tr_l}^\vee } \tau_l^* (\omega_{(e_l,v),i_l}) \big) \wedge \cdots \wedge \big( \iota_{\dd{t_{v,1}} \wedge \nu_{\tr_1}^\vee } \tau_1^*(\omega_{(e_1,v), i_1}) \big),$$ 
	where $\tau_j : \prod_{j=1}^l \vec{\mathbf{E}}_{\tr_j} \times \simplex_{\val(v)} \times W_{e,v} \rightarrow \vec{\mathbf{E}}_{\tr_j} \times W_{e_j,v}$ is induced by taking product of the projection $\prod_{j=1}^l \vec{\mathbf{E}}_{\tr_j} \rightarrow \vec{\mathbf{E}}_{\tr_j}$ with $\tau_{j} : \simplex_{\val(v)} \times W_{e,v} \rightarrow W_{e_j,v}$ (here we abuse the notation) given by $\tau_{j}(t_{v,l}, \cdots ,t_{v,1} , x) = \actionmap_{t_{v,j}}(x)$. In this case we have $V_{e,v} = \bigcap_{j=1}^l \tau_j^{-1}(V_{e_j,v})$;
	
	\item for an edge $e$ numbered by $ij$ with $\partial_{in}(e) = v_0$ and $\partial_{o}(e) = v_1$ with $v_1$ not being the outgoing vertex $v_o$, we apply the Lemma \ref{lem:homotopywkb} by taking $\zeta_S = \mathfrak{m}^{(e,v_0)}_{\vec{\chi},\vec{\varkappa}}$ (and shrinking $W_{e,v_0}$ if necessary) together with its WKB approximation, therefore we obtain the WKB approximation for $\zeta_E = \mathfrak{m}^{(e,v_1)}_{\vec{\chi},\vec{\varkappa}}$ in a neighborhood $\vec{\mathbf{E}}_{\tr_{e,v_1}} \times W_{e,v_1}$ for some small neighborhood $W_{e,v_1}$ of $x_{\ftree,e, v_1}$. In this case we have $V_{e,v_1} = \bigcup_{t \in \real} \varsigma_{t}(V_{e,v_0}) \cap \big( \vec{\mathbf{E}}_{\tr_{e,v_1}} \times W_{e,v_1} \big)$ where $\varsigma_{t}$ here is $t$-time flow of $\nabla f_{ij}/ |\nabla f_{ij}|^2$ extended to $\vec{\mathbf{E}}_{\tr_{e,v_1}} \times (M \setminus \text{Crit}(f_{ij}))$ by taking product with $\vec{\mathbf{E}}_{\tr_{e,v_1}}$;
	\item for the outgoing edge $e_o$ with outgoing vertex $v_o$, we simply take the WKB expansion of $\mathfrak{m}^{(e_o,v_o)}_{\vec{\chi},\vec{\varkappa}}$ to be that of $\mathfrak{m}^{(e_o,v_r)}_{\vec{\chi},\vec{\varkappa}}$. In this case we have $V_{e_o,v_o} = V_{e_o,v_r}$. 
	\end{enumerate}

Having the WKB approximation of $\mathfrak{m}^{(e_o,v_o)}_{\vec{\chi},\vec{\varkappa}}$, together with that for $$\frac{\ast e^{-2\lam f_{0k}} \phi_{0k}}{\| e^{-\lam f_{0k}} \phi_{0k}\|^2} \sim \frac{\lam^{\deg(q_{0k})/2}}{\|  e^{-\lam f_{0k}} \phi_{0k}\|^2} e^{-\lam \g^-_{0k}} \big(\ast \omega_{0k,0} + \ast \omega_{0k,1} \lam^{-\half} + \cdots \big)$$
from Lemma \ref{lem:eigenwkb} (here we abbreviated $\g_{q_{0k}}^-$ and $\omega_{q_{0k},i}$'s by $\g_{0k}^-$ and $\omega_{0k,i}$'s respectively), we obtain
\begin{equation}\label{eqn:wkb_for_computing_integral}
\int_{\simplex_{\tr} \times M} \mathfrak{m}^{\tr}_{\vec{\chi}_{\ftree},\vec{\varkappa}_{\ftree}} \wedge \frac{\ast e^{-2\lam f_{0k}}\phi_{0k}}{\| e^{-\lam f_{0k}} \phi_{0k}\|^2} = \frac{\lam^{r_{e_o,v_o} + \deg(q_{0k})/2}}{\| e^{-\lam f_{0k}} \phi_{0k}\|^2} \int_{\simplex_\tr \times M} e^{-\lam (\g_{e_o,v_o} + \g_{0k}^-)} \omega_{(e_o,v_o),0} \wedge \ast \omega_{0k,0} + \mathcal{O}(\lam^{-\half}).
\end{equation}

\subsubsection{Explicit computation of the integral}\label{sec:explicit_contribution}
From the generic assumption of $\vec{f}$ in Definition \ref{def:moduli_space_as_intersection}, we notice that all the points $\mathbf{t}_{\ftree,v} \in \text{int}(\simplex_{\val(v)})$. In the above WKB construction, by shrinking $\mathbf{E}_v$'s and $W_{e,v}$'s if necessary, we may always assume that $\pi_{e,v} : \vec{\mathbf{E}}_{\tr_{e,v}} \times W_{e,v}  \rightarrow V_{e,v}$ being identified with a neighborhood of zero section in the normal bundle $NV_{e,v}$ in $\vec{\mathbf{E}}_{\tr_{e,v}} \times W_{e,v}$. We notice that the element $\nu_{\tr_{e,v}} \wedge \vol_{e,v}$ (Here $\vol_{e,v}$ is introduced in Definition \ref{def:sign_of_morse_counting} as element in $\bigwedge^* T^*M_{x_{\ftree,e,v}}$) is a top degree element in $\bigwedge^* NV_{e,v}^*$, serves as an orientation in the normal direction (by extending to whole $V_{e,v}$).


We show inductively along gradient tree $\ftree$ that the integration along fiber $$(\pi_{e,v})_* \big(\lam^{r_{e,v}} e^{-\lam \g_{e,v}}\omega_{(e,v),0} \big) = 1 + \mathcal{O}(\lam^{-\half})$$ at the point $(\vec{\mathbf{t}}_{\ftree_{e,v}}, x_{\ftree,e,v})$ (here $x_{\ftree,e,v}$ is introduced in Notation \ref{not:gradient_flow_tree_notation}) in $V_{e,v}$ (Here $(\pi_{e,v})_*$ refers integration along fibers of $\pi_{e,v}$ with respect to orientation $\nu_{\tr_{e,v}} \wedge \vol_{e,v}$) using techniques from \cite[Section 3]{klchan-leung-ma}. Since $\g_{e,v}$ is non-negative Bott-Morse function with zero set $V_{e,v}$, using the well known stationary phase expansion (see e.g. \cite{dimassi1999spectral} or \cite[Lemma 58]{klchan-leung-ma}) we notice the leading order in $\lam^{-\half}$ in above integral only depend on the values of $\omega_{(e,v),0}$ at $(\vec{\mathbf{t}}_{\ftree_{e,v}}, x_{\ftree,e,v})$, and can be computed inductively as follows (we use the same notations as in the inductive WKB construction in earlier Section \ref{sec:wkb_approximation_method}):
	\begin{enumerate}
		\item for the incoming edges $e_{ij}$ with $\partial_{o}(e_{ij}) = v$, this is exactly Lemma \ref{lem:eigenwkbcal};
		
		\item for $(e,v)$ with $\partial_{in}(e) = v$ with $v$ is labeled with $1$, with subtree $\tr_2 ,\tr_1$ and outgoing edges $e_2,e_1$ ending at $v$, we have $V_{e,v}  = (V_{e_2,v} \times \vec{\mathbf{E}}_{\tr_1})\cap (V_{e_1,v} \times \vec{\mathbf{E}}_{\tr_1})$ and we can compute 
		$$
		(\pi_{e,v})_* (\lam^{r_{e,v}} e^{-\lam \g_{e,v}} \omega_{(e,v),0})= (-1)^{\varepsilon}  (\pi_{e_2,v})_* (\lam^{r_{e_2,v}} e^{-\lam \g_{e_2,v}} \omega_{(e_2,v),0}) (\pi_{e_1,v})_* (\lam^{r_{e_1,v}} e^{-\lam \g_{e_1,v}} \omega_{(e_1,v),0}) =1
		$$
		at the point $(\vec{\mathbf{t}}_{\ftree_{e,v}}, x_{\ftree,e,v})$ in $V_{e,v}$ modulo error $\mathcal{O}(\lam^{-\half})$ ($\varepsilon$ as in (2) Definition \ref{def:cut_off_operations}); 
		
		
		\item when we have $v$ labeled with $u$, we let $\tr_l,\dots,\tr_1$ be subtrees with outgoing edges $e_l,\dots,e_1$ ending at $v$ with $e_l,\dots,e_1,e$ clockwisely oriented, we notice that $V_{e,v} = \bigcap_{j=1}^l \tau_j^{-1} (V_{e_j,v})$ from WKB construction in previous Section \ref{sec:wkb_approximation_method}. From the induction, we can compute the integral $(\pi_{e_j,v})_* \big(\lam^{r_{e_j,v}} e^{-\lam \tau_j^*(\g_{e_j,v}) } \tau_j^* (\omega_{(e_j,v),0}) \big) = 1 + \mathcal{O}(\lam^{-1})$ as function on $\tau_j^{-1}((\vec{\mathbf{t}}_{\ftree_{e_j,v}},x_{\ftree,e_j,v}))$ if we identify a neighborhood $\tau_j^{-1}(\vec{\mathbf{E}}_{\tr_{j}} \times W_{e_j,v})$ of $\tau_j^{-1} (V_{e_j,v})$ with a neighborhood of zero section in the pull back normal bundle $\tau^{-1}_j (NV_{e_j,v})$ as treat $\pi_{e_j,v} : \tau^{-1}_j (NV_{e_j,v}) \rightarrow \tau_j^{-1} (V_{e_j,v})$ as integration along fibers. We obtain the identity
		$$
		(\pi_{e,v})_* (\lam^{r_{e,v}} e^{-\lam \g_{e,v}} \omega_{(e,v),0}) = \prod_{j=1}^l (\pi_{e_j,v})_* \big(\lam^{r_{e_j,v}} e^{-\lam \tau_j^*(\g_{e_j,v}) } \tau_j^* (\omega_{(e_j,v),0}) \big) =1,
		$$
		at $(\vec{\mathbf{t}}_{\ftree_{e,v}}, x_{\ftree,e,v})$ modulo error $\mathcal{O}(\lam^{-\half})$;
		
		\item for an edge $e$ numbered by $ij$ with $\partial_{in}(e) = v_0$ and $\partial_{o}(e) = v_1$ with $v_1$ not being the outgoing vertex $v_o$, we can compute $(\pi_{e,v_1})_* (\lam^{r_{e,v_1}} e^{-\lam \g_{e,v_1}} \omega_{(e,v_1),0}) = 1 + \mathcal{O}(\lam^{-\half})$ at the point $(\vec{\mathbf{t}}_{\ftree_{e,v_1}}, x_{\ftree,e,v_1})$ using the fact that $(\pi_{e,v_0})_* (\lam^{r_{e,v_0}} e^{-\lam \g_{e,v_0}} \omega_{(e,v_0),0}) = 1 + \mathcal{O}(\lam^{-\half})$ at the point $(\vec{\mathbf{t}}_{\ftree_{e,v_0}}, x_{\ftree,e,v_0})$ by applying Lemma \ref{lem:wkb_integration_relation} with $x_S =x_{\ftree,e,v_0}$ an $x_E = x_{\ftree,e,v_1}$ (notice that $\vec{\mathbf{t}}_{\ftree_{e,v_0}} = \vec{\mathbf{t}}_{\ftree_{e,v_1}}$);
		\item for the outgoing edge $e_o$ with outgoing vertex $v_o$, since we have $V_{e_o,v_o}$ and $\vec{\mathbf{E}}_{\tr} \times V_{0k}^-$ intersecting transversally at $(\vec{\mathbf{t}}_{\ftree},x_{\ftree,e_o,x_r})$, we can compute
		\begin{align*}
		&\frac{\lam^{r_{e_o,v_o} + \deg(q_{0k})/2}}{\| e^{-\lam f_{0k}} \phi_{0k}\|^2} \int_{\simplex_\tr \times M} e^{-\lam (\g_{e_o,v_o} + \g_{0k}^-)} \omega_{(e_o,v_o),0} \wedge \ast \omega_{0k,0}  \\
		=&\pm (\pi_{e_o,v_o})_* (\lam^{r_{e_o,v_o}} e^{-\lam \g_{e_o,v_o}} \omega_{(e_o,v_o),0}) \big( \frac{\lam^{\frac{\deg(q_{0k})}{2}}}{\| e^{-\lam f_{0k}} \phi_{0k}\|^2} \int_{NV^-_{x_{\ftree,e_o,x_r}}}  e^{-\lam \g_{0k}^-}  \ast \omega_{0k,0}\big) + \mathcal{O}(\lam^{-\half})\\
		=& \pm 1 + \mathcal{O}(\lam^{-\half})
		\end{align*}
		where the $\pm$ sign depending on whether the sign of gradient flow tree $\ftree$ obtained by comparing $\vol_{e_{o},v_r} \wedge \ast \vol_{q_{0k}}$ with $\vol_{M}$ as described in Definition \ref{def:sign_of_morse_counting}.
	\end{enumerate}

As a conclusion, we have proven that 
$$
\int_{M} m_{k,\tr}(\lam) (\phi_{(k-1)k},\dots,\phi_{01}) \wedge \frac{\ast e^{-2\lam f_{0k}}\phi_{0k}}{\| e^{-\lam f_{0k}} \phi_{0k}\|^2} = \sum_{i=1}^d (-1)^{\chi(\ftree_i)} + \mathcal{O}(\lam^{-\half})
$$
and hence Theorem \ref{thm:main_theorem}.

\bibliographystyle{amsplain}
\bibliography{geometry}

\end{document}